\documentclass[review,3p,square]{elsarticle}

\usepackage{amssymb}
\usepackage{amsmath}
\usepackage{amsthm}
\usepackage{stmaryrd}
\usepackage{bbm}
\usepackage{enumitem}

\usepackage{fancyhdr}
\usepackage{hyperref}
\hypersetup{%
  pdftitle={BSDEs},%
  pdfsubject={BSDEs},%
  pdfauthor={ShengJun FAN, Ying Hu},%
  pdfkeywords={BSDEs},%
  pdfstartview=FitH,%
  CJKbookmarks=true,%
  bookmarksnumbered=true,%
  bookmarksopen=true,%
  colorlinks=true, linkcolor=blue, urlcolor=blue, citecolor=blue, %
}

\usepackage{cleveref}
\crefname{thm}{Theorem}{Theorems}
\crefname{pro}{Proposition}{Propositions}
\crefname{lem}{Lemma}{Lemmas}
\crefname{rmk}{Remark}{Remarks}
\crefname{cor}{Corollary}{Corollaries}
\crefname{dfn}{Definition}{Definitions}
\crefname{ex}{Example}{Examples}
\crefname{section}{Section}{Sections}
\crefname{subsection}{Subsection}{Subsections}

\newcommand{\eps}{\varepsilon}

\newcommand{\To}{\rightarrow}
\newcommand{\as}{{\rm d}\mathbb{P}\times{\rm d} t-a.e.}

\newcommand{\ps}{\mathbb{P}-a.s.}

\newcommand{\essinf}{\mathop{\operatorname{ess\,inf}}}

\newcommand{\F}{\mathcal{F}}
\newcommand{\E}{\mathbb{E}}

\newcommand{\R}{{\mathbb R}}
\newcommand{\p}{{\mathbb P}}
\newcommand{\Q}{{\mathbb Q}}
\newcommand{\Sbb}{\mathbb{S}}
\newcommand{\Sbf}{\mathbf{S}}

\newcommand{\scal}{\mathcal{S}}
\newcommand{\lcal}{\mathcal{L}}
\newcommand{\mcal}{\mathcal{M}}

\newcommand{\hcal}{\mathcal{H}}

\newcommand{\T}{[0,T]}

\newcommand{\RE}{\forall}

\newcommand {\Dis}{\displaystyle}

\newtheorem{thm}{Theorem}[section]
\newtheorem{lem}[thm]{Lemma}
\newtheorem{pro}[thm]{Proposition}
\newtheorem{rmk}[thm]{Remark}
\newtheorem{cor}[thm]{Corollary}

\newtheorem{ex}[thm]{Example}

\journal{arXiv}
\begin{document}
\begin{frontmatter}

\title{{Uniqueness of adapted solutions to scalar BSDEs\\ with  Peano-type generators}\tnoteref{found}}
\tnotetext[found]{This work is supported by National Natural Science Foundation of China (Nos. 12171471, 12031009 and 11631004), by Key Laboratory of Mathematics for Nonlinear Sciences (Fudan University), Ministry of Education, Handan Road 220, Shanghai 200433, China; by Lebesgue Center of Mathematics ``Investissements d'avenir" program-ANR-11-LABX-0020-01, by CAESARS-ANR-15-CE05-0024 and by MFG-ANR-16-CE40-0015-01.
\vspace{0.2cm}}

\author[Fan]{Shengjun Fan} \ead{shengjunfan@cumt.edu.cn}
\author[Hu]{Ying Hu} \ead{ying.hu@univ-rennes1.fr}
\author[Tang]{Shanjian Tang} \ead{sjtang@fudan.edu.cn} \vspace{-0.5cm}

\affiliation[Fan]{organization={School of Mathematics, China University of Mining and Technology},
            city={Xuzhou 221116},
            country={China}}

\affiliation[Hu]{organization={Univ. Rennes, CNRS, IRMAR-UMR6625},
            city={F-35000, Rennes},
            country={France}}

\affiliation[Tang]{organization={Department of Finance and Control Sciences, School of Mathematical Sciences, Fudan University},
            city={Shanghai 200433},
            country={China}}

\begin{abstract}
A Backward Stochastic Differential Equation (BSDE) with a Peano-type generator,  is known to have infinitely many solutions when the terminal value is vanishing, and is shown to have possibly multiple solutions even when the  terminal value is not vanishing  but nonnegative. In this paper, we study the uniqueness of adapted solutions  of such a BSDE when the terminal value is almost surely positive. Two methods are developed. The first one is to connect the BSDE to an optimal  stochastic control problem: under suitable integrability of the terminal values, with a verification argument, we prove that the first component of the adapted solution pair is the value process for the  optimal  stochastic  control problem. The second one appeals to a change of variables, and is more inclined to analysis: by a change of variables, the original BSDE is reduced to a convex quadratic BSDE, and then using the $\theta$-difference method, we give a sharp result in some special case,  which includes  the BSDE governing the well-known Kreps-Porteus utility.\vspace{0.2cm}
\end{abstract}

\begin{keyword}
Backward stochastic differential equation \sep Existence and uniqueness \sep Peano-type\sep \\
\hspace*{1.8cm} $\theta$-difference method \sep Kreps-Porteus utility.\vspace{0.2cm}

\MSC[2010] 60H10
\end{keyword}

\end{frontmatter}
\vspace{-0.4cm}

\section{Introduction}
\label{sec:1-Introduction}
\setcounter{equation}{0}

Define $\R:=(-\infty,+\infty)$, $\R_+:=[0,+\infty)$ and $\R_{++}:=(0,+\infty)$. Given a real $T>0$ and a positive integer $d$. Let $(\Omega, \F, \mathbb{P})$ be a complete probability space with augmented filtration $(\F_t)_{t\in\T}$
generated by a $d$-dimensional standard Brownian motion $(B_t)_{t\in\T}$, and $\F=\F_T$. The equalities and inequalities between random elements are always understood to hold $\ps$. We are concerned with the following scalar backward stochastic differential equation (BSDE for short):
\begin{equation}\label{eq:1.1}
  Y_t=\xi+\int_t^T g(s,Y_s,Z_s){\rm d}s-\int_t^T Z_s {\rm d}B_s, \ \ t\in\T,
\end{equation}
where $\xi$ is an $\F_T$-measurable real-valued random variable, called the terminal value, and for each $(y,z)\in \R\times\R^{1\times d}$, the stochastic process
$g(\cdot, \cdot, y, z):\Omega\times\T\to \R $
is $(\F_t)$-progressively measurable (the function $g$ is called the generator). An adapted solution of BSDE \eqref{eq:1.1} is an pair of $(\F_t)$-progressively measurable processes $(Y_t,Z_t)_{t\in\T}$ taking values in $\R\times\R^{1\times d}$ such that $\ps$, $t\mapsto Y_t$ is continuous, $t\mapsto |g(t,Y_t,Z_t)|+|Z_t|^2$ is integrable, and \eqref{eq:1.1} is fulfilled. \citet{PardouxPeng1990SCL} established the first existence and uniqueness result on the square-integrable adapted solution of a nonlinear BSDE under the Lipschitz condition on the generator $g$ and the square-integrable condition on the terminal value $\xi$. Since then, BSDEs have received a lot of attentions and gradually become a powerful tool in various fields such as mathematical finance, nonlinear expectation, PDE, stochastic control and game, and so on. The reader is referred to \cite{ElKarouiPengQuenez1997MF,Peng1997BSDEbook,
Pardoux1999Nonlinear,Kobylanski2000AP,ChenEpstein2002Econometrica,
CoquetHuMeminPeng2002PTRF,Peng2004LectureNotes,
HuImkeller2005AAP,Jia2008PHDThesis,Jiang2008AAP,Xing2017FS,FanHu2021SPA,FanHuTang2023arXiv} for details among others.

The existence of a square-integrable adapted solution for BSDE \eqref{eq:1.1} was given in \citet{LepeltierSanMartin1997SPL} when the terminal value $\xi$ is square-integrable, and the generator $g$ is only continuous and has a linear growth in both unknown variables $(y,z)$. See also \citet{FanJiang2012JAMC} for the case of the $L^p\ (p>1)$ solution. By virtue of the localization method of \citet{BriandHu2006PTRF} along with the a priori estimate technique and the comparison theorem for adapted solutions of BSDEs, existence were given for more general situations that the generator $g$ may have a one-sided linear/superlinear growth in the state variable $y$ and a super-linear/sub-quadratic/quadratic growth in the state variable $z$. Relevant results are available in for example \cite{Kobylanski2000AP,BriandHu2006PTRF,BarrieuElKaroui2013AoP,
Fan2016SPA,Fan2016SPL,FanHu2021SPA,FanHuTang2023SPA,FanHuTang2023arXiv}. It is well known that under the above-mentioned conditions, BSDE \eqref{eq:1.1} can have multiple adapted solutions when the terminal value is zero. For example, for each constant $c\in\T$, the pair of processes
$$
(Y_t,Z_t):=\frac{1}{4}\left([(c-t)^+]^2,0\right), \ \ t\in\T
$$
is an adapted solution of the following BSDE:
\begin{equation}\label{eq:1.2}
Y_t=\int_t^T \sqrt{|Y_s|}~{\rm d}s-\int_t^T Z_s {\rm d}B_s,\ \ t\in\T.
\end{equation}
This is because  the following Ordinary Differential Equation (ODE)
$$\left\{
\begin{array}{l}
y'(t)=-\sqrt{y(t)},\ \ t\in\T;\\
y(T)=0
\end{array}
\right.
$$
has infinitely many solutions: $y_t=[(c-t)^+]^2/4$ on the time interval $\T$ for constants $c\in \T$.
To the best of our knowledge, Peano first explored this kind of phenomenon of multiple solutions of ODEs (see for example \citet{Herrmann2001CRD}, \citet{Crippa2011JDE} and \citet{Delarue2014DCDS}). We call the generator like $\sqrt{y~}$ being of Peano-type.

To obtain the uniqueness of an adapted solution of BSDE \eqref{eq:1.1}, some further conditions have to be imposed on the generator $g$. The uniform Lipschitz condition used in \citet{PardouxPeng1990SCL} is just the classical one. Weaker conditions than the uniform Lipschitz condition were subsequently
 explored and verified during the last three decades. Among these conditions we would like to especially mention Mao's non-Lipschitz condition, the monotonicity condition and the one-sided Osgood condition with respect to the state variable $y$, and the uniform continuity condition and the convexity/concavity condition with respect to the state variable $z$. The reader is referred to \cite{Mao1995SPA,Pardoux1999Nonlinear,Kobylanski2000AP, BriandDelyonHu2003SPA,BriandHu2008PTRF,Jia2008CRA,Jia2010SPA,
FanJiangTian2011SPA,DelbaenHuRichou2011AIHPPS,BarrieuElKaroui2013AoP,
BahlaliEssakyHassani2015SIAM,
Fan2016SPA,BahlaliEddahbiOuknine2017AoP,FanHuTang2020CRM,FanHu2021SPA,
FanHuTang2023SPA,FanHuTang2023arXiv} for more details. We mention that in order to verify the validity of these conditions, some novel ideas and tools were also constantly presented and developed, such as Gronwall's and Bihari's inequalities, Girsanov's transform, the $\theta$-difference method and the test function technique.

The  Peano-type generator like $\sqrt{y~}$ goes beyond all the above-mentioned works. The previous BSDE \eqref{eq:1.2} admits multiple adapted solutions. Moreover, even for a  nonzero but nonnegative terminal value,  a BSDE with a Peano-type generator might still admit multiple adapted solutions. For example, for a stopping time $\tau\in (0,T)$, let $A\in\F_\tau$ be an event such that $\p(A)\in (0,1)$, and $\lambda\in [0,1]$ be a constant. Then $\eta_\lambda:=\lambda\tau+(1-\lambda)T$ is an $\F_\tau$-measurable random variable. Denote by  $(Y^\lambda_t,Z^\lambda_t)_{t\in [0,\tau]}$  the adapted square-integrable maximal solution of the following BSDE
$$
Y_t^\lambda=\left(1+\frac{T-\tau}{2}\right)^2{\bf 1}_{A}+\left[\frac{(\eta_\lambda-\tau)^+}{2}\right]^2{\bf 1}_{A^c}+\int_t^\tau \sqrt{Y^\lambda_s}~{\rm d}s-\int_t^\tau Z^\lambda_s {\rm d}B_s,\ \ t\in [0,\tau].\vspace{0.1cm}
$$
Then, by setting for each $t\in\T$,
$$
\bar Y_t^\lambda:=\left\{\left(1+\frac{T-t}{2}\right)^2{\bf 1}_{A}+\left[\frac{(\eta_\lambda-t)^+}{2}\right]^2{\bf 1}_{A^c}\right\}{\bf 1}_{t>\tau}+Y^\lambda_t {\bf 1}_{t\leq \tau}\ \ \ {\rm and}\ \ \ \bar Z_t^\lambda:=Z^\lambda_t{\bf 1}_{t\leq \tau},
$$
we can directly verify that $(\bar Y_t^\lambda,\bar Z_t^\lambda)_{t\in\T}$ is a square-integrable adapted solution of
the following BSDE
$$
Y_t={\bf 1}_{A}+\int_t^T \sqrt{Y_s}~{\rm d}s-\int_t^T Z_s {\rm d}B_s,\ \ t\in\T.
$$
Note the constant $\lambda$ can be arbitrarily chosen in $[0,1]$. The last BSDE admits infinitely many square-integrable adapted solutions. Now, let us take a look at the ODE with a Peano-type generator when the terminal value is almost surely positive. By virtue of the variable separation method, it is straightforward to verify  that for any given $c>0$, the following ODE
$$\left\{
\begin{array}{l}
y'(t)=-\sqrt{y(t)},\ \ t\in\T;\\
y(T)=c
\end{array}
\right.
$$
does have a unique solution $y_t=(\sqrt{c}+T/2-t/2)^2$ on the time interval $\T$. Then, a question is naturally asked: for any given square-integrable terminal value $\xi\in\R_{++}$, is the square-integrable solution $(Y_t,Z_t)_{t\in\T}$ such that $Y_\cdot\in\R_{++}$ of the following BSDE
\begin{equation}\label{eq:1.3}
Y_t=\xi+\int_t^T \sqrt{Y_s}~{\rm d}s-\int_t^T Z_s {\rm d}B_s,\ \ t\in\T.
\end{equation}
necessarily unique? When $\xi\geq c>0$ for some constant $c$, it is not difficult to see that the answer is affirmative. In fact, by taking the conditional mathematical expectation we have $Y_t\geq c$ for any adapted square-integrable positive solution $(Y_t,Z_t)_{t\in T}$ of BSDE \eqref{eq:1.3}. Then, $(Y_\cdot,Z_\cdot)$ has to be the unique square-integrable adapted solution of the following BSDE with a Lipschitz generator:
$$
Y_t=\xi+\int_t^T \left(\sqrt{Y_s}~{\bf 1}_{\{Y_s\geq c\}}+\sqrt{c}~{\bf 1}_{\{Y_s< c\}}\right)~{\rm d}s-\int_t^T Z_s {\rm d}B_s,\ \ t\in\T.
$$
But, when $\essinf\xi=0$, it seems that the answer of the previous question can not be easily obtained by any existing results. The objective of the present paper is to solve this problem. Two kinds of more complicate scenarios are further investigated. To the best of our knowledge, this is the first study on the uniqueness of the adapted solutions of BSDEs with generators of Peano-type.

More specifically, in \cref{sec:3-Mainresult} of the present paper we apply a verification method to study uniqueness of the adapted solutions of BSDEs with  Peano-type generators when the terminal value $\xi$ is almost surely positive. In fact, it is proved that under some integrability of the terminal value, the first component $Y_\cdot$ of the adapted solution pair is indeed the value process for an optimal stochastic control problem. See \cref{thm:3.2} for more details. The second is an analytical approach. Through a change of variables, the original BSDE is connected to one convex BSDE with quadratic growth in \cref{sec:4-FurtherDiscussion}. Then we use a $\theta$-difference method to get a sharp result in some special case which can be applied to the BSDE governing the Kreps-Porteus utility. See \cref{thm:4.1} for more details. The proofs are quite involved. The whole ideas consist in employing the classical dual theory on the convex functional and the $\theta$-difference method employed initially in \citet{BriandHu2008PTRF}. The Fenchel-Moreau theorem, H\"{o}lder's inequality, Jensen's inequality, It\^{o}-Tanaka's formula, Girsanov's transform and Doob's maximal inequality on sub-martingales play crucial roles in our proof. For more clarity, let us take BSDE \eqref{eq:1.3} for an example to quickly illustrate our whole idea.

We now introduce the verification method used in the proof of \cref{thm:3.2}.  First, we have  the duality,
\begin{equation}\label{eq:1.4*}
\sqrt{y}=\inf\limits_{q>0} \left\{qy+\frac{1}{4q}\right\},\quad y\in\R_+.
\end{equation}
Then we define the following  subset of $(\F_t)$-progressively measurable $\R_{++}$-valued processes:
$$
\hcal:=\left\{(q_t)_{t\in\T}:\
\begin{array}{l}
q_\cdot \ {\rm is\ an}\ (\F_t)\hbox{-}{\rm progressively\ measurable}\ \R_{++}\hbox{-}{\rm valued\ process}\vspace{0.1cm}\\
{\rm such\ that}\ \ \E\left[{\rm e}^{2\int_0^T q_s{\rm d}s}\right]<+\infty\ \ {\rm and}
\ \ \E\left[\left(\int_0^T \frac{1}{q_t}{\rm d}t\right)^2\right]<+\infty
\end{array}
\right\}.\vspace{0.1cm}
$$
For any $q_\cdot\in \hcal$ and square-integrable $\xi$, let $(Y^{q_\cdot}_t,Z^{q_\cdot}_t)_{t\in\T}$ be the unique square-integral adapted solution of the following BSDE
\begin{equation}\label{eq:1.5*}
Y_t^{q_\cdot}=\xi+\int_t^T \left(q_s Y_s^{q_\cdot}+\frac{1}{4q_s}\right)~{\rm d}s-\int_t^T Z_s^{q_\cdot} {\rm d}B_s,\ \ t\in\T.\vspace{0.1cm}
\end{equation}
We have
$$
Y^{q_\cdot}_t=\E\left[\left.{\rm e}^{\int_t^Tq_s{\rm d}s}\xi+\int_t^T {\rm e}^{\int_t^sq_u{\rm d}u}\frac{1}{4q_s} {\rm d}s\right|\F_t\right],\ \ t\in\T.
$$
Finally, suppose that both $\xi$ and $1/\sqrt{\xi}$ are square-integrable and $(Y_t,Z_t)_{t\in\T}$ is any square-integrable adapted solution of BSDE \eqref{eq:1.3} such that $Y_\cdot\in\R_{++}$, and we prove the following representation
\begin{equation}\label{eq:1.6*}
Y_t=\essinf\limits_{q_\cdot\in\hcal}Y_t^{q_\cdot}, \quad t\in \T,
\end{equation}
which yields the uniqueness of the adapted solution of BSDE \eqref{eq:1.3}.
In fact, in view of  \eqref{eq:1.3}-\eqref{eq:1.5*}, by a comparison argument,  we see that for each $q\in\hcal$ and $t\in\T$,
$$Y_t\leq Y^{q_\cdot}_t.$$
Now, from the duality we pick
\begin{equation}\label{eq:1.7*}
q^*_t:=\frac{1}{2\sqrt{Y_t}},\ \  t\in\T.
\end{equation}
It remains to prove that $q^*_\cdot\in\hcal$ and $Y_\cdot=Y^{q^*_\cdot}_\cdot$. For this purpose, we first give some a priori estimate on $Y_\cdot$. It follows from It\^{o}'s formula that
$$
2\sqrt{Y_t}=2\sqrt{\xi~}+\int_t^T \left(1+\frac{|Z_s|^2}{4(Y_s)^{3/2}}\right){\rm d}s-\int_t^T \frac{Z_s}{\sqrt{Y_s}} {\rm d}B_s,\ \ t\in\T.
$$
Thus, for each $t\in\T$,
$$
2\sqrt{Y_t}\geq \E\left[2\sqrt{\xi~}|\F_t\right]+T-t
$$
and then
\begin{equation}\label{eq:1.8*}
Y_t\geq \left(\E\left[\sqrt{\xi~}|\F_t\right]+\frac{T-t}{2}\right)^2
\geq \left(\eta+\frac{T-t}{2}\right)^2
\end{equation}
with
$$\eta:=\min_{t\in\T}\E\left[\sqrt{\xi~}|\F_t\right].$$
From this a priori estimate, we can now finish the desired proof. First, by \eqref{eq:1.7*} we know that $q^*_\cdot$ is an $(\F_t)$-progressively measurable $\R_{++}$-valued process and $\E\left[\left(\int_0^T 1/q^*_t{\rm d}t\right)^2\right]<+\infty$. Then, by the a priori estimate \eqref{eq:1.8*} we obtain
$$
\int_0^T \frac{1}{\sqrt{Y_t}}{\rm d}t\leq \int_0^T \frac{2}{2\eta+T-t}{\rm d}t=2\ln \left(\frac{2\eta +T}{2\eta}\right)=\ln \left(1+\frac{T}{2\min\limits_{t\in\T}\E\left[\sqrt{\xi~}|\F_t\right]}\right)^2
$$
and, in light of Jensen's inequality and Doob's maximal inequality,
\begin{equation}\label{eq:1.9*}
\begin{array}{lll}
\Dis \E\left[\exp\left(\int_0^T \frac{1}{\sqrt{Y_t}}{\rm d}t\right)\right]&\leq &\Dis \E\left[ \max\limits_{t\in\T}\left(1+\frac{T}{2\E\left[\sqrt{\xi~}
|\F_t\right]}\right)^2\right]\\[3mm]
&\leq &\Dis \E\left[ \max\limits_{t\in\T}\left(\E\left[\left.1+\frac{T}{2\sqrt{\xi~}
}\right|\F_t\right]\right)^2\right]\\[3mm]
&\leq &\Dis 4\E\left[ \left(1+\frac{T}{2\sqrt{\xi~}
}\right)^2\right]<+\infty.
\end{array}
\end{equation}
Combining \eqref{eq:1.7*} and \eqref{eq:1.9*} yields that
$\E\left[{\rm e}^{2\int_0^T q^*_s{\rm d}s}\right]<+\infty$. Hence, $q^*_\cdot\in\hcal$. Finally, in light of the duality \eqref{eq:1.4*} along with \eqref{eq:1.7*}, we have
\begin{equation}\label{eq:1.10*}
\sqrt{Y_s}=q^*_s Y_s+\frac{1}{4 q^*_s},\ \ s\in\T.
\end{equation}
Moreover, in light of \eqref{eq:1.3} and \eqref{eq:1.10*}, from the uniqueness of the square-integrable adapted solution of BSDE \eqref{eq:1.5*} it follows that
$Y_\cdot=Y^{q^*_\cdot}_\cdot$. Thus, we have proved that the positive square-integrable adapted solution of BSDE \eqref{eq:1.3} is unique when $1/\xi$ is integrable. We remark that this verification method is applied in this paper to a more general situation that the generator $g$ contains not only a nonnegative term being of Peano-type in the state variable $y$, but also a nonnegative term satisfying the monotonicity condition in $y$ and a real-valued term being uniformly Lipschitz continuous in the state variables $(y,z)$.

In the sequel, let us further illustrate the $\theta$-difference method used in the proof of \cref{thm:4.1}. Suppose that $(Y_t,Z_t)_{t\in\T}$ is a square-integrable adapted solution of BSDE \eqref{eq:1.3} such that $Y_\cdot\in \scal^p(\R_{++})$ for some $p>1$. Define $\bar\xi:=2\sqrt{\xi~}$ and for each $t\in\T$,
$$
\bar Y_t:=2\sqrt{Y_t}\ \ \ {\rm and}\ \ \ \bar Z_t:=\frac{Z_t}{\sqrt{Y_t}}.
$$
By virtue of It\^{o}'s formula, we conclude that $(Y_\cdot,Z_\cdot)$ is an adapted solution of BSDE \eqref{eq:1.3} such that $Y_\cdot\in \scal^p(\R_{++})$ if and only if $(\bar Y_\cdot, \bar Z_\cdot)$ is an adapted solution of the following BSDE:
\begin{equation}\label{eq:1.11}
\bar Y_t=\bar\xi+\int_t^T \left(1+\frac{|\bar Z_s|^2}{2\bar Y_s}\right){\rm d}s-\int_t^T \bar Z_s {\rm d}B_s,\ \ t\in\T
\end{equation}
such that $\bar Y_\cdot\in \scal^{2p}(\R_{++})$. Note that the generator $g(y):=\sqrt{y~}$ of BSDE \eqref{eq:1.3} is a concave function defined on $\R_+$, while the generator $\bar g(y,z):=1+|z|^2/2y$ of BSDE \eqref{eq:1.11} is a jointly convex function defined on $\R_{++}\times \R^{1\times d}$. Then, for each $\theta\in (0,1)$ and $(y_1,y_2,z_1,z_2)\in \R_{++}\times \R_{++}\times \R^{1\times d}\times \R^{1\times d}$, we have
$$
\begin{array}{l}
\Dis {\bf 1}_{\{y_1>\theta y_2\}}\left(\bar g(y_1,z_1)-\theta \bar g(y_2,z_2)\right)\vspace{0.1cm}\\
\ \ \Dis ={\bf 1}_{\{y_1>\theta y_2\}}\left[\bar g\left(\theta y_2+(1-\theta)\frac{y_1-\theta y_2}{1-\theta},~\theta z_2+(1-\theta)\frac{z_1-\theta z_2}{1-\theta}\right)-\theta \bar g(y_2,z_2)\right]\vspace{0.1cm}\\
\ \ \Dis \leq (1-\theta){\bf 1}_{\{y_1>\theta y_2\}}\bar g\left(\left(\frac{y_1-\theta y_2}{1-\theta}\right)^+, \frac{z_1-\theta z_2}{1-\theta}\right).
\end{array}
$$
Based on the above observation, we can employ the $\theta$-difference method as in \citet{BriandHu2008PTRF}, \citet{FanHu2021SPA}, and \citet{FanHuTang2023SPA,FanHuTang2023arXiv} to verify that BSDE \eqref{eq:1.11} admits a unique adapted solution $(Y_\cdot,Z_\cdot)$ such that $Y_\cdot\in \scal^{2p}(\R_{++})$. As a result, $(Y_\cdot,Z_\cdot)$ is the unique adapted solution of BSDE \eqref{eq:1.3} such that $\bar Y_\cdot\in \scal^p(\R_{++})$. This method is applied to a slightly more general situation in this paper. See BSDE \eqref{eq:4.1} and \cref{thm:4.1} in \cref{sec:4-FurtherDiscussion} for more details. We remark that  in \cref{sec:4-FurtherDiscussion}, although the BSDE is special, but the result is sharper since no integrability  is imposed on $1/\xi$.

The rest of this paper is organized as follows. Some notations, spaces and two technical lemmas used later are presented in \cref{sec:2-Preliminaries}. In \cref{sec:3-Mainresult} we state and prove \cref{thm:3.2} as  the first main result, with the help of a verification argument. Finally in \cref{sec:4-FurtherDiscussion}, using the $\theta$-difference method, we study the uniqueness of the adapted solution of a special BSDE, and give \cref{thm:4.1}. Several remarks and examples are included to illustrate both theorems. See \cref{rmk:3.1,rmk:3.3,rmk:4.1*} along with examples \ref{ex:3.5} and \ref{ex:4.2} for details. In particular, the uniqueness of the Kreps-Porteus utility of an endowment $\xi\in L^p(\R_{++})$ is demonstrated in \cref{ex:4.2}.

\section{Preliminaries}
\label{sec:2-Preliminaries}
\setcounter{equation}{0}

\subsection{Notations and spaces}

For $a,b\in \R$, define $a\wedge b:=\min\{a,b\}$, $a^+:=\max\{a,0\}$, and ${\rm sgn}(x):={\bf 1}_{x>0}-{\bf 1}_{x\leq 0}$ with ${\bf 1}_A$ being the indicator function of set $A$. For each $p>1$, denote by $L^p(\F_t;\R)$ for $t\in\T$ the total of $\F_t$-measurable $\R$-valued random variables $\xi$ satisfying $\E[|\xi|^p]<+\infty$ (in particular, $L^p(\F_T;\R)$ is also denoted by $L^p(\R)$ for simplicity), $\lcal^p(\R)$ the total of $(\F_t)$-progressively measurable $\R$-valued processes $(X_t)_{t\in\T}$ satisfying
$$
\|X\|_{\lcal^p(\R)}:=\left\{\E\left[\left(\int_0^T |X_t|{\rm d}t\right)^p\right]\right\}^{{1\over p}}<+\infty,
$$
$\scal^p(\R)$ the total of $(\F_t)$-adapted continuous $\R$-valued processes $(Y_t)_{t\in\T}$ satisfying
$$\|Y\|_{{\scal}^p(\R)}:=\left(\E\left[\sup_{t\in\T} |Y_t|^p\right]\right)^{{1\over p}}<+\infty,$$
and $\mcal^p(\R^{1\times d})$ the total of $(\F_t)$-progressively measurable $\R^{1\times d}$-valued processes $(Z_t)_{t\in\T}$ satisfying
$$
\|Z\|_{\mcal^p(\R^{1\times d})}:=\left\{\E\left[\left(\int_0^T |Z_t|^2{\rm d}t\right)^{p/2}\right] \right\}^{{1\over p}}<+\infty.\vspace{0.1cm}
$$
If the $\R$ is replaced with $\R_+$ or $\R_{++}$ in anyone of the above spaces, then it represents the total of those elements in the original space taking values in $\R_+$ or $\R_{++}$.


Denote by $\Sbf$ the totality of concave functions $\rho(\cdot):\R_+\To \R_+$ such that $\rho(0)=0$, $\rho(x)>0$ for $x>0$ and $\rho'(\cdot)$ is a continuous function on $(0,+\infty)$ taking values in $\R_+$. Note that for $\rho(\cdot)\in\Sbf$, we have
$$
\RE \lambda\in (0,1),\ \RE x\geq 0,\ \ \ \ \rho(\lambda x)=\rho(\lambda x+(1-\lambda)0)\geq \lambda\rho(x)+(1-\lambda)\rho(0)=\lambda\rho(x)
$$
and then the function $\rho(x)/x,\ x>0$ is nonincreasing, i.e., for each $x,y>0$,
\begin{equation}\label{eq:2.1}
{\rm if}\ x\leq y, \ \ {\rm then}\ \  \frac{\rho(x)}{x}=\frac{\rho( \frac{x}{y} y)}{x}\geq \frac{\rho(y)}{y}.
\end{equation}
Hence, $\rho(x)\leq \rho(1)x$ for $x\geq 1$, and then $\rho(\cdot)$ is of linear growth, i.e.,
\begin{equation}\label{eq:2.2}
\RE x\in\R_+,\ \ \rho(x)\leq \rho(1)+\rho(1)x.
\end{equation}
Furthermore, by \eqref{eq:2.1} we have for each $x,y\in \R_{++}$,
$$
\frac{\rho(x+y)}{x+y}\leq \frac{\rho(x)}{x},\ \ \frac{\rho(x+y)}{x+y}\leq \frac{\rho(y)}{y}
$$
and
$$
\rho(x+y)=\frac{\rho(x+y)}{x+y}x+\frac{\rho(x+y)}{x+y}y\leq \rho(x)+\rho(y).
$$
As a result, $\rho(\cdot)$ can be dominated by itself, i.e.,
\begin{equation}\label{eq:2.3}
\RE x,y\in \R_+,\ \ |\rho(x)-\rho(y)|\leq \rho(|x-y|).
\end{equation}

\subsection{Two Lemmas}

By virtue of Fenchel-Moreau theorem, we can prove the following lemma which gives a dual representation of function $\rho(\cdot)$ in $\Sbf$. It will play an important role in the proof of our main result.

\begin{lem}\label{lem:2.1}
For $\rho(\cdot)\in \Sbf$, let
\begin{equation}\label{eq:2.4}
\rho^*(q):=\sup_{x\in\R_+}\{\rho(x)-qx\},\ \ q\in \R_+.
\end{equation}
Then, $\rho^*(\cdot)$ is a nonincreasing convex function on $\R_+$ taking values in $\R_+\cup \{+\infty\}$, and
\begin{equation}\label{eq:2.5}
\rho(x)=\inf_{q\in\R_+,\rho^*(q)\in\R_+}\{\rho^*(q)+qx\},\ \ x\in\R_+.
\end{equation}
Moreover, for any $x_0\in \R_{++}$, by picking $q=\rho'(x_0)\in\R_+$ we have $\rho^*(q)\in\R_+$ and
\begin{equation}\label{eq:2.6}
\rho(x_0)=\rho^*(q)+qx_0.
\end{equation}
\end{lem}

\begin{proof}
By \eqref{eq:2.4}, it is straightforward to verify that $\rho^*(\cdot)$ is a nonincreasing convex function on $\R_+$ taking values in $\R_+\cup \{+\infty\}$. Now, let us verify \eqref{eq:2.5}. Define
\begin{equation}\label{eq:2.7}
\bar\rho(x):=\left\{
\begin{array}{rl}
-\rho(x), & x\geq 0;\\
+\infty, & x<0.
\end{array}
\right.
\end{equation}
Then, $\bar\rho(\cdot)$ is a proper convex function defined on $\R$. By Fenchel-Moreau theorem, if
\begin{equation}\label{eq:2.8}
\bar\rho^*(q):=\sup_{x\in\R}\{-qx-\bar\rho(x)\},\ \ q\in\R,
\end{equation}
then
\begin{equation}\label{eq:2.9}
\bar\rho(x)=\sup_{q\in\R}\{-qx-\bar\rho^*(q)\},\ \ x\in\R.
\end{equation}
It follows from \eqref{eq:2.7} and \eqref{eq:2.8} that
$$
\bar\rho^*(q)=\sup_{x\geq 0}\{-qx+\rho(x)\},\ \ q\in\R,
$$
and then, in light of \eqref{eq:2.4}, $\bar\rho^*(q)=\rho^*(q)$ for $q\geq 0$ and $\bar\rho^*(q)=+\infty$ for $q<0$. As a result, by \eqref{eq:2.9},
$$
\bar\rho(x)=\sup_{q\geq 0}\{-qx-\bar\rho^*(q)\}=\sup_{q\geq 0}\{-qx-\rho^*(q)\},\ \ x\in\R.
$$
Hence, in light of \eqref{eq:2.7}, we obtain
$$
\rho(x)=-\bar\rho(x)=-\sup_{q\geq 0}\{-qx-\rho^*(q)\}=\inf_{q\geq 0}\{qx+\rho^*(q)\},\ \ x\geq 0,
$$
which is just \eqref{eq:2.5}. Moreover, for each $x_0\in \R_{++}$, by picking $q=\rho'(x_0)=-\bar\rho'(x_0)\geq 0$ we have
$$
\bar\rho(x_0)=-qx_0-\bar\rho^*(q),
$$
which yields $\rho^*(q)=\bar\rho^*(q)\in\R_+$ and \eqref{eq:2.6}. The proof is complete.
\end{proof}

\begin{rmk}\label{rmk:2.1*}
For example, for each given $\alpha\in (0,1)$, the function $\rho(x):=x^\alpha,\ x\in\R_+$ belongs to $\Sbf$. If $$
\rho^*(q):=\sup_{x\in\R_+}\{\rho(x)-qx\}=\sup_{x\in\R_+}\{x^\alpha-qx\}=(1-\alpha)
\left(\frac{\alpha}{q}\right)^{\frac{\alpha}{1-\alpha}},\ \ \ q\in \R_+,
$$
then $\rho^*(\cdot)$ is a nonincreasing convex function on $\R_+$ taking values in $\R_+\cup \{+\infty\}$, and
$$
\rho(x)=\inf_{q\in\R_+,\rho^*(q)\in\R_+}\{\rho^*(q)+qx\}=\inf_{q\in\R_{++}}\left\{(1-\alpha)
\left(\frac{\alpha}{q}\right)^{\frac{\alpha}{1-\alpha}}+qx\right\},\ \ \ x\in\R_+.
$$
Moreover, for any $x_0\in \R_{++}$, by picking $q=\rho'(x_0)=\alpha x_0^{\alpha-1}\in\R_+$ we have $\rho^*(q)=(1-\alpha) x_0^\alpha\in\R_+$ and
$$
\rho(x_0)=\rho^*(q)+qx_0=x_0^\alpha.
$$
\end{rmk}

It is noteworthy that for $\rho\in\Sbf$, in light of \eqref{eq:2.1} and \eqref{eq:2.3}, if $\lim\limits_{x\To 0^+}\rho(x)/x=k$ for some $k>0$, then $\rho(\cdot)$ must satisfy the Lipschitz condition with Lipschitz constant $k$, and if $\lim\limits_{x\To 0^+}\rho(x)/x=+\infty$ and
$$
\int_{0^+} \frac{1}{\rho(x)}{\rm d}x:=\lim\limits_{\eps\To 0^+}\int_{\eps} \frac{1}{\rho(x)}{\rm d}x=+\infty,
$$
then $\rho(\cdot)$ must satisfy the Osgood condition. In the present paper, we are interested in the other case of $\rho\in \Sbf$, i.e., $\rho(\cdot)$ belongs to the following set:
\begin{equation}\label{eq:2.10}
\Sbb:=\left\{\rho(\cdot)\in \Sbf\  \Big|\  \rho'(0+0):=\lim\limits_{x\To 0^+} \frac{\rho(x)}{x}=+\infty\ \ {\rm and}\ \ \int_{0^+} \frac{1}{\rho(x)}{\rm d}x<+\infty\right\}.
\end{equation}
For example, let $k,c>0$, $\alpha\in (0,1)$ and $\beta\geq 1$ be four given constants, and $\eps>0$ be a sufficiently small constant. Define $\rho_1(x):=kx,\ x\in\R_+$; and the following nine functions:
$$
\rho_2(x):=\left\{
\begin{array}{ll}
x|\ln x|^\beta, & x\in [0,\eps];\vspace{0.1cm}\\
\rho_2(\eps)+\rho'_2(\eps-0)(x-\eps), & x\in (\eps,+\infty),
\end{array}
\right.
$$

$$
\rho_3(x):=\left\{
\begin{array}{ll}
x|\ln x||\ln|\ln x||^\beta, & x\in [0,\eps];\vspace{0.1cm}\\
\rho_3(\eps)+\rho'_3(\eps-0)(x-\eps), & x\in (\eps,+\infty),
\end{array}
\right.
$$

$$
\rho_4(x):=\left\{
\begin{array}{ll}
x|\ln x|^{\alpha}, & x\in [0,\eps];\vspace{0.1cm}\\
\rho_4(\eps)+\rho'_4(\eps-0)(x-\eps), & x\in (\eps,+\infty),
\end{array}
\right.
$$

$$
\rho_5(x):=\left\{
\begin{array}{ll}
x|\ln x||\ln|\ln x||^{\alpha}, & x\in [0,\eps];\vspace{0.1cm}\\
\rho_5(\eps)+\rho'_5(\eps-0)(x-\eps), & x\in (\eps,+\infty),
\end{array}
\right.
\ \ \ \ \ \ \rho_6(x):=kx^{\alpha},\ x\in\R_+,
$$

$$
\rho_7(x):=\left\{
\begin{array}{ll}
kx^{\alpha}, & x\in [0,c];\vspace{0.1cm}\\
kc^{\alpha}+k\alpha c^{{\alpha}-1}(x-c), & x\in (c,+\infty),
\end{array}
\right.\ \ \
\rho_8(x):=\left\{
\begin{array}{ll}
\sqrt{x}, & x\in [0,1];\vspace{0.1cm}\\
-\frac{1}{4}x^2+x+\frac{1}{4}, & x\in (1,2);\vspace{0.1cm}\\
\frac{5}{4}, & x\in [2,+\infty),
\end{array}
\right.
$$

$$
\rho_9(x):=\left\{
\begin{array}{ll}
x^{\alpha}|\ln x|^k, & x\in [0,\eps];\vspace{0.1cm}\\
\rho_9(\eps)+\rho'_9(\eps-0)(x-\eps), & x\in (\eps,+\infty),
\end{array}
\right.
$$
and
$$
\rho_{10}(x):=\left\{
\begin{array}{ll}
\frac{x^{\alpha}}{|\ln x|^k}, & x\in [0,\eps];\vspace{0.1cm}\\
\rho_{10}(\eps)+\rho'_{10}(\eps-0)(x-\eps), & x\in (\eps,+\infty).
\end{array}
\right.
$$
Then, $\rho_1(\cdot)$ belongs to the above-mentioned first case, both $\rho_2(\cdot)$ and $\rho_3(\cdot)$ belong to the second case, and all functions $\rho_i(\cdot), i=4,\ldots,10,$ belong to $\Sbb$. Functions in $\Sbb$ will be called to be  Peano-type.

We especially emphsis that it is straightforward to verify that for each $\rho(\cdot)\in\Sbb$ and $h(\cdot)\in\Sbf$, the function $\bar\rho(\cdot):=\rho(\cdot)+h(\cdot)$ must belong to $\Sbb$.

 Functions $\rho(\cdot)$ in $\Sbb$ have the following properties which will be used later, see \eqref{eq:3.20}-\eqref{eq:3.21} in \cref{sec:3-Mainresult}.

\begin{lem}\label{lem:2.2}
For $\rho(\cdot)\in \Sbb$, let $\rho_c(x):=\rho(x)+c$ for $x\in\R_+$ and some constant $c\in\R_+$, and let
$$
H_c(u):=\int_0^u \frac{1}{\rho_c(x)}{\rm d}x=\int_0^u \frac{1}{c+\rho(x)}{\rm d}x,\ \ u\in\R_+.\vspace{0.1cm}
$$
Then, $H_c(\cdot)$ is well-defined and strictly increasing on $\R_+$, $H_c(0)=0$ and $\lim\limits_{u\To +\infty} H_c(u)=+\infty$. Moreover, for each $k_1\geq H_c(1)$ and $k_2,k_3>0$, we have
\begin{equation}\label{eq:2.11*}
H_c^{-1}(k_1+k_2)\leq 2{\rm e}^{k_2(\rho(1)+c)}H_c^{-1}(k_1)
\end{equation}
and
\begin{equation}\label{eq:2.11}
\rho_c(k_3H_c^{-1}(k_1+k_2))\leq 2{\rm e}^{k_2(\rho(1)+c)} \rho_c(k_3H_c^{-1}(k_1)),\vspace{0.1cm}
\end{equation}
where $H_c^{-1}(\cdot):\R_+\To \R_+$ denotes the inverse function of $H_c(\cdot)$.
\end{lem}

\begin{proof}
By the definition \eqref{eq:2.10} of $\Sbb$, it is obvious that $H_c(\cdot)$ is well-defined and a strictly increasing function on $\R_+$ with $H_c(0)=0$. In light of \eqref{eq:2.2}, we have
$$
\lim\limits_{u\To +\infty} H_c(u)=\int_0^{+\infty}\frac{1}{c+\rho(x)}{\rm d}x\geq \frac{1}{\rho(1)+c}\int_0^{+\infty}\frac{1}{1+x}{\rm d}x=+\infty.
$$
Then the inverse function $H_c^{-1}(\cdot):\R_+\To \R_+$ of $H_c(\cdot)$ is well-defined and is also strictly increasing on $\R_+$. Now, fix $k_1\geq H_c(1)$ and $k_2,k_3>0$, and let $u_2=H_c^{-1}(k_1+k_2)$ and $u_1=H_c^{-1}(k_1)$. Then, $u_2>u_1\geq 1$ and in light of \eqref{eq:2.2}, we have
$$
\begin{array}{lll}
k_2=(k_1+k_2)-k_1&=&\Dis H_c(u_2)-H_c(u_1)=\int_{u_1}^{u_2}\frac{1}{c+\rho(x)}{\rm d}x\\
&\geq & \Dis \frac{1}{\rho(1)+c}\int_{u_1}^{u_2}\frac{1}{1+x}{\rm d}x=\frac{1}{\rho(1)+c}\ln\frac{1+u_2}{1+u_1}=
\frac{1}{\rho(1)+c}\ln\left(1+\frac{u_2-u_1}{1+u_1}\right),
\end{array}
$$
which yields that
$$
u_2-u_1\leq \left({\rm e}^{k_2(\rho(1)+c)}-1\right)(1+u_1)\leq 2\left({\rm e}^{k_2(\rho(1)+c)}-1\right)u_1.
$$
Thus, picking $C=2{\rm e}^{k_2(\rho(1)+c)}>1$, we obtain $u_2\leq Cu_1$. This is just \eqref{eq:2.11*}. Finally, in light of \eqref{eq:2.1} and $C>1$, we have
$$
\rho_c(k_3H_c^{-1}(k_1+k_2))=\rho_c(k_3u_2)\leq \rho_c(Ck_3u_1)\leq C\rho_c(k_3u_1)=C \rho_c(k_3H_c^{-1}(k_1)),
$$
which is the desired conclusion \eqref{eq:2.11}. The proof is complete.
\end{proof}

\begin{rmk}\label{rmk:2.3}
The preceding proof also manifests that the conclusions of \cref{lem:2.2} are still true when the condition of $\rho'(0+0)=+\infty$ in \eqref{eq:2.10} is dropped from the definition of $\Sbb$.
\end{rmk}

\section{Main result}
\label{sec:3-Mainresult}
\setcounter{equation}{0}

Let us always fix several constants $\beta,\bar\beta,\tilde\beta,\gamma\geq 0$, $\lambda>0$ and $p>1$ along with two $(\F_t)$-progressively measurable processes $(\alpha_t)_{t\in\T}$ and $(\bar\alpha_t)_{t\in\T}$ in $\lcal^p(\R_+)$. Let $q$ be the conjugate of $p$, i.e., $1/p+1/q=1$.

\subsection{Statement of the main result}

Suppose that the function $g$ can be written as
\begin{equation}\label{eq:3.1}
g(\omega,t,y,z):=f(\omega,t,y)+\bar f(\omega,t,y)+\tilde f(\omega,t,y,z),
\end{equation}
where the three functions
$$
\begin{array}{l}
f(\omega,t,y):\Omega\times\T\times\R_+\longrightarrow \R_+,\vspace{0.1cm}\\
\bar f(\omega,t,y):\Omega\times\T\times\R_+\longrightarrow \R_+,\vspace{0.1cm}\\
\tilde f(\omega,t,y,z):\Omega\times\T\times\R_+\times\R^{1\times d}\longrightarrow \R\vspace{0.1cm}
\end{array}
$$
are all $(\F_t)$-progressively measurable for each $(y,z)\in \R\times\R^{1\times d}$, and satisfy the following assumptions:
\begin{enumerate}
\renewcommand{\theenumi}{(A\arabic{enumi})}
\renewcommand{\labelenumi}{\theenumi}

\item\label{A1} $\as$, $f(\omega,t,\cdot)$ is continuous, nonnegative and concave on $\R_+$, and has  the partial differential  $f'(\omega,t,y)$  at any point $y\in \R_{++}$ in the last argument. Moreover, there is a function $\varphi(\cdot)\in\Sbb$ such that $\as$,
\begin{itemize}
\item $\RE y\in \R_+,\ \ \alpha_t(\omega)+\varphi(y)\leq f(\omega,t,y)\leq \bar\alpha_t(\omega)+\beta y$,
\item $\RE y \in \R_{++}, \ \ 0\leq f'(\omega,t,y)\leq \lambda \varphi'(y)$.
\end{itemize}

\item\label{A2} $\as$, $\bar f(\omega,t,\cdot)$ is continuous and nonnegative on $\R_+$, and satisfies a monotonicity condition and a linear growth condition, i.e.,
    \begin{itemize}
    \item $\RE (y_1,y_2)\in \R_+\times\R_+,\ \ {\rm sgn}(y_1-y_2)(\bar f(\omega,t,y_1)-\bar f(\omega,t,y_2))\leq \bar\beta |y_1-y_2|$,\vspace{0.1cm}
    \item $\RE y\in \R_+,\ \ 0\leq \bar f(\omega,t,y)\leq \bar\alpha_t(\omega)+\bar\beta y.$
    \end{itemize}

\item\label{A3} $\as$, $\tilde f(\omega,t,\cdot,\cdot)$ is continuous on $\R_+\times\R^{1\times d}$ taking values on $\R$, $\tilde f(\omega,t,0,0)=0$, and satisfies a uniform Lipschitz condition, i.e., for each $(y_1,y_2,z_1,z_2)\in \R_+\times\R_+\times\R^{1\times d}\times\R^{1\times d}$,
    $$
    |\tilde f(\omega,t,y_1,z_1)-\tilde f(\omega,t,y_2,z_2)|\leq \tilde\beta |y_1-y_2|+\gamma |z_1-z_2|.
    $$
\end{enumerate}

\begin{rmk}\label{rmk:3.1}
Regarding assumptions \ref{A1}-\ref{A3}, we make the following remarks.
\begin{itemize}
\item [(i)] Let $(u_t)_{t\in\T}$, $(\bar u_t)_{t\in\T}$ and $(\tilde u_t)_{t\in\T}$ be three $(\F_t)$-progressively measurable nonnegative bounded processes such that $\essinf_{t\in\T} u_t\geq c_1$ and $\essinf_{t\in\T} \bar u_t\geq c_2$ for some two constants $c_1,c_2>0$, and $\phi(\cdot)$ be a function in $\Sbb$. Then, it is not difficult to verify that the following function
    $$
    f(\omega,t,y):=\alpha_t(\omega)+u_t(\omega)\phi(\bar u_t(\omega)y+\tilde u_t(\omega)),\ \ (\omega,t,y)\in\Omega\times\T\times\R_+
    $$
    satisfies assumption \ref{A1} with $\varphi(\cdot):=c_1\phi(c_2\cdot)$. In particular, any function in $\Sbb$ satisfies \ref{A1}. Consequently, a generator $f$ satisfying \ref{A1} has a close link with a function of Peano-type.

\item [(ii)] For each $(\omega,t,y)\in \Omega\times\T\times\R_+$, define
$$
\bar f_1(\omega,t,y):=(1-\sqrt{y})^+\ \ \ {\rm and}\ \
\bar f_2(\omega,t,y):=1+\sin y+(3-|B_t(\omega)|y)^+.
$$
It is obvious that neither $g_1(\omega,t,\cdot)$ nor $g_2(\omega,t,\cdot)$ satisfy the usual uniform Lipschitz condition, while both of them satisfy assumption \ref{A2}.

\item [(iii)] We emphasize that $\as$, the function $g(\omega,t,\cdot,\cdot)$ is not necessarily concave/convex, and that the function $\tilde f$ in \ref{A3} may take negative values.
\end{itemize}
\end{rmk}

The main result of the paper is stated as follows.

\begin{thm}\label{thm:3.2}
Suppose the function $g$ defined in \eqref{eq:3.1} satisfies assumptions \ref{A1}-\ref{A3} and $\xi\in L^p(\R_{++})$. Let the process $\alpha_\cdot\geq c$ for a constant $c\geq 0$ and the constants $\hat p>\bar p>q$, and let $\tilde p>1$ be the constant such that $1/p+1/\bar p=1/\tilde p$. If $\frac{1}{\varphi(\xi)+c}\in L^{\hat p\lambda {\rm e}^{2\tilde\beta T}}(\R_+)$, then the following BSDE
\begin{equation}\label{eq:3.2}
  Y_t=\xi+\int_t^T g(s,Y_s,Z_s){\rm d}s-\int_t^T Z_s {\rm d}B_s, \ \ t\in\T
\end{equation}
admits a unique adapted solution $(Y_t,Z_t)_{t\in\T}\in \scal^p(\R_{++})\times \mcal^p(\R^{1\times d})$. Moreover, by denoting
\begin{equation}\label{eq:3.3}
f^*(\omega,t,q):=\sup_{y\in\R_+}(f(\omega,t,y)-qy),\ \ (\omega,t,q)\in\Omega\times \T\times\R_+,
\end{equation}
we have for each $t\in\T$,
\begin{equation}\label{eq:3.4}
Y_t=\essinf_{q_\cdot\in \hcal} \big({\rm e}^{-\int_0^t q_s{\rm d}s}Y^{q_\cdot}_t\big),
\end{equation}
and there exists $q^*_\cdot:=f'(\cdot,Y_\cdot)\in\hcal$ such that $Y_\cdot=Y^{q^*}_\cdot,$
where the admissible control set is defined by
\begin{equation}\label{eq:3.5}
\hcal:=\left\{(q_t)_{t\in\T}:\
\begin{array}{l}
q_\cdot \ {\rm is\ an}\ (\F_t)\hbox{-}{\rm progressively\ measurable}\ \R_{+}\hbox{-}{\rm valued\ process}\vspace{0.1cm}\\
 {\rm such\ that}\ \ {\rm e}^{\int_0^T q_s{\rm d}s}\in L^{\bar p}(\R_+)\ \ {\rm and}
 \ \ (f^*(t,q_t))_{t\in\T}\in \lcal^p(\R_+)
\end{array}
\right\}
\end{equation}
and $(Y^{q_\cdot}_t,Z^{q_\cdot}_t)_{t\in\T}$ is the unique adapted solution in $\scal^{\tilde p}(\R_{++})\times\mcal^{\tilde p}(\R^{1\times d})$ of the following BSDE:
\begin{equation}\label{eq:3.6}
  Y^{q_\cdot}_t=\xi^{q_\cdot}+\int_t^T g^{q_\cdot}(s,Y^{q_\cdot}_s,Z^{q_\cdot}_s) {\rm d}s-\int_t^T Z^{q_\cdot}_s {\rm d}B_s, \ \ t\in\T
\end{equation}
with $\xi^{q_\cdot}:={\rm e}^{\int_0^T q_s{\rm d}s}\xi$ and for each $(\omega,t,y,z)\in\Omega\times\T\times\R_+\times\R^{1\times d}$,
\begin{equation}\label{eq:3.7}
g^{q_\cdot}(t,y,z):={\rm e}^{\int_0^t q_s{\rm d}s}\left[f^*(t,q_t)+\bar f(t,{\rm e}^{-\int_0^t q_s{\rm d}s}y\big)+\tilde f(t,{\rm e}^{-\int_0^t q_s{\rm d}s}y, {\rm e}^{-\int_0^t q_s{\rm d}s}z\big)\right].
\end{equation}
\end{thm}

\begin{rmk}\label{rmk:3.3}
We make the following remarks.
\begin{itemize}
\item [(i)] Suppose the function $g$ defined in \eqref{eq:3.1} satisfies  assumptions \ref{A1}-\ref{A3}. Let $\bar g$ be defined as follows:
\begin{equation}\label{eq:3.7-1}
\bar g(\omega,t,y,z):=\left\{
\begin{array}{ll}
g(\omega,t,y,z), & (\omega,t,y,z)\in \Omega\times\T\times\R_+\times\R^{1\times d};\vspace{0.1cm}\\
g(\omega,t,0,z), & (\omega,t,y,z)\in \Omega\times\T\times(-\infty,0)\times\R^{1\times d}.
\end{array}
\right.
\end{equation}
It is not hard to verify that for each $(y,z)\in\R\times\R^{1\times d}$, $\bar g(\omega,t,y,z)$ is $(\F_t)$-progressively measurable taking values on $\R$, and $\as$,
\begin{itemize}
\item $\bar g(\omega,t,\cdot,\cdot)$ is continuous on $\R\times\R^{1\times d}$,\  and\ \ $0\leq \alpha_t(\omega)\leq \bar g(\omega,t,0,0)\leq 2\bar\alpha_t(\omega)$\vspace{0.1cm};
\item $\RE (y,z)\in \R\times\R^{1\times d},\ \ |\bar g(\omega,t,y,z)|\leq 2\bar\alpha_t(\omega)+ (\beta+\bar\beta+\tilde\beta)|y|+\gamma |z|$.\vspace{0.1cm}
\end{itemize}
From some existing results such as Theorem 1 of \citet{Fan2016SPL}, Theorem 5.1 of \citet{Fan2016SPA} and Theorem 2.3 (iv) of \citet{FanHuTang2023arXiv}, it follows that for each $\xi\in L^p(\R)$, the following BSDE \begin{equation}\label{eq:3.7-2}
  Y_t=\xi+\int_t^T \bar g(s,Y_s,Z_s){\rm d}s-\int_t^T Z_s {\rm d}B_s, \ \ t\in\T
\end{equation}
admits an adapted solution $(Y_t,Z_t)_{t\in\T}\in \scal^p(\R)\times \mcal^p(\R^{1\times d})$. And, by the classical comparison theorem under the Lipschitz condition we have that if $\xi>0$, then $Y_\cdot>0$. This means that for each $\xi\in L^p(\R_{++})$, BSDE \eqref{eq:3.2} admits an adapted solution $(Y_t,Z_t)_{t\in\T}\in \scal^p(\R_{++})\times \mcal^p(\R^{1\times d})$.

\item [(ii)] Suppose that the function $f$ disappears in the definition \eqref{eq:3.1} of the function $g$, and functions $\bar f$ and $\tilde f$ satisfy assumptions \ref{A2} and \ref{A3}. Let $\bar g$ be defined in \eqref{eq:3.7-1}. Then for each $(y,z)\in\R\times\R^{1\times d}$, $\bar g(\omega,t,y,z)$ is $(\F_t)$-progressively measurable taking values on $\R$, and $\as$,
\begin{itemize}
\item $\bar g(\omega,t,\cdot,\cdot)$ is continuous on $\R\times\R^{1\times d}$,\  and\ \ $0\leq \bar g(\omega,t,0,0)\leq \bar\alpha_t(\omega)$\vspace{0.1cm};
\item $\RE (y_1,y_2,z)\in \R\times\R\times\R^{1\times d},\ \ {\rm sgn}(y_1-y_2)(\bar g(\omega,t,y_1,z)-\bar g(\omega,t,y_2,z))\leq (\bar\beta+\tilde\beta)|y_1-y_2|$\vspace{0.1cm};
\item $\RE (y,z_1,z_2)\in \R\times\R^{1\times d}\times\R^{1\times d},\ \ |\bar g(\omega,t,y,z_1)-\bar g(\omega,t,y,z_2)|\leq \gamma |z_1-z_2|$\vspace{0.1cm};
\item $\RE y\in \R_+,\ \ |\bar g(\omega,t,y,0)|\leq \bar\alpha_t(\omega)+(\bar\beta+\tilde\beta)|y|.$
\end{itemize}
It then follows from Theorem 4.2 of \citet{BriandDelyonHu2003SPA} that for each $\xi\in L^p(\R)$, BSDE \eqref{eq:3.7-2} admits a unique adapted solution $(Y_t,Z_t)_{t\in\T}\in \scal^p(\R)\times \mcal^p(\R^{1\times d})$. And, by the classical comparison theorem under the Lipschitz condition we have that if $\xi>0$, then $Y_\cdot>0$. This means that for each $\xi\in L^p(\R_{++})$, BSDE \eqref{eq:3.2} admits a unique adapted solution $(Y_t,Z_t)_{t\in\T}\in \scal^p(\R_{++})\times \mcal^p(\R^{1\times d})$.

\item [(iii)] Suppose the function $g$ defined in \eqref{eq:3.1} satisfies the above assumptions \ref{A1}-\ref{A3}. Suppose further that $\xi\in L^p(\R_{++})$ is bigger than or equal to some constant $\bar c>0$. For each $c\in (0,\bar c]$, define
$$
\hat f(\omega,t,y):=\left\{
\begin{array}{ll}
f(\omega,t,y), & (\omega,t,y)\in \Omega\times\T\times[c,+\infty);\vspace{0.1cm}\\
\frac{f(\omega,t,c)}{c}y, & (\omega,t,y)\in \Omega\times\T\times [0,c).
\end{array}
\right.
$$
By assumption \ref{A1} we can easily verify that $\as$, for each $(y_1,y_2)\in \R_+\times\R_+$,
$$
|\hat f(\omega,t,y_1)-\hat f(\omega,t,y_2)|\leq \lambda\varphi'(c)|y_1-y_2|.
$$
Thus, by replacing $f$ with $\hat f$ in the definition \eqref{eq:3.1} of the function $g$ and using a similar argument as in (ii), we can conclude that BSDE \eqref{eq:3.2} admits at most an adapted solution $(Y_t,Z_t)_{t\in\T}\in \scal^p(\R_{+})\times \mcal^p(\R^{1\times d})$ such that $Y_\cdot \geq c$. On the other hand, note that the following BSDE
$$
Y_t=\bar c+\int_t^T (-\tilde\beta |Y_s|-\gamma |Z_s|){\rm d}s-\int_t^T Z_s {\rm d}B_s, \ \ t\in\T
$$
admits a unique adapted solution $(\bar c{\rm e}^{\tilde\beta(t-T)},0)_{t\in\T}$ belonging to $\scal^p(\R_{++})\times \mcal^p(\R^{1\times d})$. By the nonnegativity of $f$ and $\bar f$, the Lipschitz condition of $\tilde f$ and the classical comparison theorem under the Lipschitz condition it can be checked that for any adapted solution $(Y_t,Z_t)_{t\in\T}$ of BSDE \eqref{eq:3.2}, we have $Y_\cdot\geq \bar c{\rm e}^{-\tilde\beta T}$. Consequently, BSDE \eqref{eq:3.2} admits a unique adapted solution $(Y_t,Z_t)_{t\in\T}$ in the space $\scal^p(\R_{++})\times \mcal^p(\R^{1\times d})$.
\end{itemize}
\end{rmk}

From \cref{thm:3.2}, we immediately have  the assertion (iii) of \cref{rmk:3.3} as the following corollary.

\begin{cor}\label{cor:3.4}
Suppose the function $g$ defined in \eqref{eq:3.1} satisfies assumptions \ref{A1}-\ref{A3} and $\xi\in L^p(\R_{++})$. If $\xi\geq \bar c$ for some $\bar c>0$, then BSDE \eqref{eq:3.2} admits a unique adapted solution in $\scal^p(\R_{++})\times\mcal^p(\R^{1\times d})$.
\end{cor}

At the end of this subsection, we present the following example.

\begin{ex}\label{ex:3.5}
For simplicity, we only consider the case of $\bar f(\cdot)=\tilde f(\cdot)\equiv 0$ appearing in \eqref{eq:3.1}. For each $(\omega,t,y)\in \Omega\times\T\times\R_+$, define
$$
g_1(\omega,t,y):=\sqrt{y+|B_t(\omega)|}, \ \ g_2(\omega,t,y):=\sqrt{y}+|B_t(\omega)|,\ \
g_3(\omega,t,y):=\sqrt{y+T-t},\ \
$$
and
$$
g_4(\omega,t,y):=\sqrt{y}+T-t,\ \ g_5(\omega,t,y):=\sqrt{y+t},\ \ g_6(\omega,t,y):=\sqrt{y}+t.
$$
It is easy to verify that all these functions satisfy the assumption \ref{A1} with $\varphi(x):=\sqrt{x}\in \Sbb$ and $\lambda=1$. It then follows from \cref{thm:3.2} that for each $i=1,\cdots, 6$ and $\xi\in L^p(\R_{++})$, if $1/\xi\in L^{2\hat p}$ for some $\hat p>q$, then BSDE \eqref{eq:3.2} with $g:=g_i$ admits a unique adapted solution $(Y_\cdot^i,Z_\cdot^i)\in \scal^p(\R_{++})\times\mcal^p(\R^{1\times d})$. For example,
for any $(\F_t)$-progressively measurable bounded processes $(a_t)_{t\in\T}\in \R$ and $(b_t)_{t\in\T}\in \R^{1\times d}$, the terminal value $\xi:=e^{\int_0^T a_t {\rm d}t|B_T| +\int_0^T b_t{\rm d}B_t}$ satisfies the above-mentioned integrability condition.

On the other hand, if we first consider the BSDE on the time interval $[T/2,T]$, then from \cref{thm:3.2} with $\varphi(x):=\sqrt{x}\in \Sbb$ and $\alpha_\cdot:\equiv \sqrt{T/2}>0$ it follows that BSDEs \eqref{eq:3.2} with $g:=g_5, g_6$ admit unique adapted solutions $(Y_t^5,Z_t^5)_{t\in [T/2,T]}$ and $(Y_t^6,Z_t^6)_{t\in [T/2,T]}$, respectively,  on the time interval $[T/2,T]$ in the space $\scal^p(\R_{++})\times\mcal^p(\R^{1\times d})$. Furthermore, by taking the conditional mathematical expectation it is straightforward to verify that for each $t\in [T/2,T]$,
$$
Y^5_t\geq \sqrt{\frac{T}{2}}(T-t)\ \ \ {\rm and}\ \ \ Y^6_t\geq \frac{T}{2}(T-t)
$$
and then
$$
Y^5_{T/2}\geq \left(\frac{T}{2}\right)^{\frac{3}{2}}\ \ \ {\rm and}\ \ \
Y^6_{T/2}\geq \left(\frac{T}{2}\right)^2.
$$
Thus, \cref{cor:3.4} yields that BSDEs \eqref{eq:3.2} with generators $g:=g_5, g_6$ admit also unique adapted solutions $(Y_t^5,Z_t^5)_{t\in [0,T/2]}$ and $(Y_t^6,Z_t^6)_{t\in [0,T/2]}$ , respectively, on the time interval $[0,T/2]$ in the space $\scal^p(\R_{++})\times\mcal^p(\R^{1\times d})$.

As a result, we conclude that for each $\xi\in L^p(\R_{++})$, without assuming any integrability  on $1/\xi$, BSDEs \eqref{eq:3.2} with generators $g:=g_5, g_6$ admit unique adapted solutions $(Y_\cdot^5,Z_\cdot^5)$ and $(Y_\cdot^6,Z_\cdot^6)$ , respectively,  in the space  $\scal^p(\R_{++})\times\mcal^p(\R^{1\times d})$.
\end{ex}

\subsection{Proof of the main result}

Now, we are at a position to prove \cref{thm:3.2}.

\begin{proof}[Proof of \cref{thm:3.2}]
The existence of an adapted solution $(Y_t,Z_t)_{t\in\T}\in \scal^p(\R_{++})\times \mcal^p(\R^{1\times d})$ of BSDE \eqref{eq:3.2} has been illustrated in \cref{rmk:3.3} (i).

Next, we further verify existence and uniqueness of the adapted solution $(Y^{q_\cdot}_t,Z^{q_\cdot}_t)_{t\in\T}\in \scal^{\tilde p}(\R_{++})\times \mcal^{\tilde p}(\R^{1\times d})$ of BSDE \eqref{eq:3.6} for any $q_\cdot\in\hcal$. In fact, let $q_\cdot\in\hcal$ and define
$$
\bar g^{q_\cdot}(\omega,t,y,z):=\left\{
\begin{array}{ll}
g^{q_\cdot}(\omega,t,y,z), & (\omega,t,y,z)\in \Omega\times\T\times\R_+\times\R^{1\times d};\vspace{0.1cm}\\
g^{q_\cdot}(\omega,t,0,z), & (\omega,t,y,z)\in \Omega\times\T\times(-\infty,0)\times\R^{1\times d}.
\end{array}
\right.
$$
In light of assumptions \ref{A2} and \ref{A3} along with the definition \eqref{eq:3.7} of $g^{q_\cdot}$, we can verify that for each $(y,z)\in\R\times\R^{1\times d}$, $\bar g^{q_\cdot}(\omega,t,y,z)$ is $(\F_t)$-progressively measurable taking values on $\R$, and $\as$,
\begin{itemize}
\item $\bar g^{q_\cdot}(\omega,t,\cdot,\cdot)$ is continuous on $\R\times\R^{1\times d}$,\  and\ \ $\bar g^{q_\cdot}(\omega,t,0,0)\geq 0$\vspace{0.1cm};
\item $\RE (y_1,y_2,z)\in \R\times\R\times\R^{1\times d},\ \ {\rm sgn}(y_1-y_2)(\bar g^{q_\cdot}(\omega,t,y_1,z)-\bar g^{q_\cdot}(\omega,t,y_2,z))\leq (\bar\beta+\tilde\beta)|y_1-y_2|$\vspace{0.1cm};
\item $\RE (y,z_1,z_2)\in \R\times\R^{1\times d}\times\R^{1\times d},\ \ |\bar g^{q_\cdot}(\omega,t,y,z_1)-\bar g^{q_\cdot}(\omega,t,y,z_2)|\leq \gamma |z_1-z_2|$\vspace{0.1cm};
\item $\RE y\in \R_+,\ \ |\bar g^{q_\cdot}(\omega,t,y,0)|\leq {\rm e}^{\int_0^t q_s(\omega){\rm d}s} \left[\bar\alpha_t(\omega)+f^*(\omega,t,q_t(\omega))\right]+(\bar\beta+\tilde\beta)|y|$.\vspace{0.1cm}
\end{itemize}
Moreover, by the definition \eqref{eq:3.5} of $\hcal$ we know that
${\rm e}^{\int_0^T q_s{\rm d}s}\in L^{\bar p}(\R_+)$ and
$f^*(\cdot,q_\cdot)\in \lcal^p(\R_+)$. It then follows from H\"{o}lder's inequality that
$$
\Dis \E\left[|\xi^{q_\cdot}|^{\tilde p}\right]=\E\left[{\rm e}^{\tilde p \int_0^T q_s{\rm d}s}|\xi|^{\tilde p}\right]\leq \left(\E\left[ {\rm e}^{\bar p \int_0^T q_s{\rm d}s}\right]\right)^{\frac{\tilde p}{\bar p}}\left(\E\left[|\xi|^p\right]\right)^{\frac{\tilde p}{p}}<+\infty
$$
and
$$
\begin{array}{lll}
\Dis \E\left[\left(\int_0^T |\bar g^{q_\cdot}(t,0,0)| {\rm d}t\right)^{\tilde p}\right]&\leq & \Dis \E\left[{\rm e}^{\tilde p \int_0^T q_s{\rm d}s}\left(\int_0^T \left(\bar\alpha_t+f^*(t,q_t)\right) {\rm d}t\right)^{\tilde p}\right]\vspace{0.1cm}\\
&\leq & \Dis \left(\E\left[ {\rm e}^{\bar p \int_0^T q_s{\rm d}s}\right]\right)^{\frac{\tilde p}{\bar p}}\left(\E\left[\left(\int_0^T \left(\bar\alpha_t+f^*(t,q_t)\right) {\rm d}t\right)^p\right]\right)^{\frac{\tilde p}{p}}\vspace{0.1cm}\\
&<& +\infty.
\end{array}
$$
Thus, an identical argument to (ii) of \cref{rmk:3.3} yields that for each $q_\cdot\in \hcal$, BSDE \eqref{eq:3.6} admits a unique adapted solution $(Y^{q_\cdot}_t,Z^{q_\cdot}_t)_{t\in\T}\in \scal^{\tilde p}(\R_{++})\times \mcal^{\tilde p}(\R^{1\times d})$.

In the sequel, for each $q_\cdot\in \hcal$, define $\hat Y^{q_\cdot}_\cdot:={\rm e}^{-\int_0^\cdot q_s{\rm d}s}Y^{q_\cdot}_\cdot$ and $\hat Z^{q_\cdot}_\cdot:={\rm e}^{-\int_0^\cdot q_s{\rm d}s}Z^{q_\cdot}_\cdot$,
and let us prove that for each $t\in\T$,
\begin{equation}\label{eq:3.8}
Y_t\leq \hat Y^{q_\cdot}_t.
\end{equation}
First, in light of \eqref{eq:3.6} and \eqref{eq:3.7}, by It\^{o}'s formula we obtain
\begin{equation}\label{eq:3.9}
\hat Y^{q_\cdot}_t=\xi+\int_t^T \hat g^{q_\cdot}(s,\hat Y^{q_\cdot}_s,\hat Z^{q_\cdot}_s){\rm d}s-\int_t^T \hat Z^{q_\cdot}_s {\rm d}B_s, \ \ t\in\T,
\end{equation}
where for each $(\omega,t,y,z)\in \Omega\times\T\times\R_+\times\R^{1\times d}$,
$$
\hat g^{q_\cdot}(\omega,t,y,z):=q_t(\omega) y+f^*(\omega,t,q_t(\omega))+\bar f(\omega,t,y)+\tilde f(\omega,t,y,z).
$$
On the other hand, in light of \eqref{eq:3.3}, by a similar analysis as in \cref{lem:2.1} we obtain
$$
f(\omega,t,y)=\inf\limits_{q_\cdot\in\R_+;f^*(\omega,t,q)\in\R_+}
\{f^*(\omega,t,q)+qy\},\ \ (\omega,t,y)\in \Omega\times\T\times\R_+,
$$
and then
$$
f(t,Y_t)\leq q_tY_t+f^*(t,q_t),\ \ t\in\T.
$$
Furthermore, in light of the definitions of $g$ and $\hat g^{q_\cdot}$, by assumptions \ref{A2} and \ref{A3} along with the last inequality we can deduce that
$$
\begin{array}{l}
\Dis {\bf 1}_{Y_t\geq \hat Y^{q_\cdot}_t}\left(g(t,Y_t,Z_t)
-g^{q_\cdot}(t,\hat Y^{q_\cdot}_t,\hat Z^{q_\cdot}_t)\right)\vspace{0.1cm}\\
\ \ \Dis ={\bf 1}_{Y_t\geq \hat Y^{q_\cdot}_t}\left[f(t,Y_t)-q_tY_t-f^*(t,q_t)+q_t(Y_t-\hat Y^{q_\cdot}_t)\right]\vspace{0.1cm}\\
\ \ \ \ \ \Dis +{\bf 1}_{Y_t\geq \hat Y^{q_\cdot}_t}\left(\bar f(t,Y_t)-\bar f(t,\hat Y^{q_\cdot}_t)\right)+{\bf 1}_{Y_t\geq \hat Y^{q_\cdot}_t}\left(\tilde f(t,Y_t,Z_t)-\tilde f(t,\hat Y^{q_\cdot}_t,\hat Z^{q_\cdot}_t)\right)\vspace{0.1cm}\\
\ \ \Dis \leq (q_t+\bar\beta+\tilde\beta)(Y_t-\hat Y^{q_\cdot}_t)^+ +\gamma {\bf 1}_{Y_t\geq \hat Y^{q_\cdot}_t} |Z_t-\hat Z^{q_\cdot}_t|,\ \ t\in\T,
\end{array}
$$
and, by letting $\hat q_t:=(\bar\beta+\tilde\beta)t+\int_0^t q_s{\rm d}s,\ t\in\T$ and using It\^{o}-Tanaka's formula for ${\rm e}^{\hat q_t}(Y_t-\hat Y^{q_\cdot}_t)^+$,
\begin{equation}\label{eq:3.10}
\Dis {\rm e}^{\hat q_t}(Y_t-\hat Y^{q_\cdot}_t)^+\leq  \Dis \int_t^T {\rm e}^{\hat q_s}{\bf 1}_{Y_s\geq \hat Y^{q_\cdot}_s} \gamma |Z_s-\hat Z^{q_\cdot}_s|{\rm d}s-\int_t^T {\rm e}^{\hat q_s}{\bf 1}_{Y_s\geq \hat Y^{q_\cdot}_s} (Z_s-\hat Z^{q_\cdot}_s){\rm d}B_s,\ \ t\in\T.
\end{equation}
Now, we define a probability measure $\Q^{q_\cdot}$ which is equivalent to $\p$ by
$$
\frac{{\rm d}\Q^{q_\cdot}}{{\rm d}\p}:=\exp\left(
\int_0^T \frac{\gamma (Z_s-\hat Z^{q_\cdot}_s)}{|Z_s-\hat Z^{q_\cdot}_s|}{\bf 1}_{\{|Z_s-\hat Z^{q_\cdot}_s|\neq 0\}}{\rm d}B_s-\frac{\gamma^2}{2}\int_0^T {\bf 1}_{\{|Z_s-\hat Z^{q_\cdot}_s|\neq 0\}} {\rm d}s\right).
$$
Then, ${\rm d}\Q^{q_\cdot}/{\rm d}\p$ admits moments of any order. According to Girsanov's theorem, the process
$$
B_t^{q_\cdot}:=B_t-\int_0^t \frac{\gamma (Z_s-\hat Z^{q_\cdot}_s)^\top}{|Z_s-\hat Z^{q_\cdot}_s|}{\bf 1}_{\{|Z_s-\hat Z^{q_\cdot}_s|\neq 0\}}{\rm d}s,\ \ t\in \T
$$
is an $(\F_t)$-adapted standard Brownian motion, and \eqref{eq:3.10} can be written as
\begin{equation}\label{eq:3.11}
{\rm e}^{\hat q_t}(Y_t-\hat Y^{q_\cdot}_t)^+\leq \Dis -\int_t^T {\rm e}^{\hat q_s}{\bf 1}_{Y_s\geq \hat Y^{q_\cdot}_s} (Z_s-\hat Z^{q_\cdot}_s){\rm d}B_s^{q_\cdot},\ \ t\in\T.
\end{equation}
Finally, by the definition of $\hat q_\cdot$ and the integrability of $Z^{q_\cdot}_\cdot$ we know that ${\rm e}^{\hat q_\cdot}\hat Z^{q_\cdot}_\cdot={\rm e}^{(\bar\beta+\tilde\beta)\cdot}Z^{q_\cdot}_\cdot\in \mcal^{\tilde p}(\R^{1\times d})$, and, by H\"{o}lder's inequality we deduce that
$$
\begin{array}{lll}
\Dis \E\left[\left(\int_0^T |{\rm e}^{\hat q_t} Z_t|^2 {\rm d}t\right)^{\frac{\tilde p}{2}}\right]&\leq & \Dis {\rm e}^{\tilde p(\bar\beta+\tilde\beta)T}\E\left[{\rm e}^{\tilde p \int_0^T q_s{\rm d}s}\left(\int_0^T |Z_t|^2 {\rm d}t\right)^{\frac{\tilde p}{2}}\right]\vspace{0.1cm}\\
&\leq & \Dis {\rm e}^{\tilde p(\bar\beta+\tilde\beta)T}\left(\E\left[ {\rm e}^{\bar p \int_0^T q_s{\rm d}s}\right]\right)^{\frac{\tilde p}{\bar p}}\left(\E\left[\left(\int_0^T |Z_t|^2 {\rm d}t\right)^{\frac{p}{2}}\right]\right)^{\frac{\tilde p}{p}}\vspace{0.1cm}\\
&<& +\infty
\end{array}
$$
and then
$$
\begin{array}{l}
\Dis \E_{\Q^{q_\cdot}}\left[\left(\int_0^T |{\rm e}^{\hat q_t} (Z_t-\hat Z^{q_\cdot}_t)|^2 {\rm d}t\right)^{\frac{1}{2}}\right]\vspace{0.1cm}\\
\ \ \Dis \leq \left(\E\left[\left(\int_0^T |{\rm e}^{\hat q_t} (Z_t-\hat Z^{q_\cdot}_t)|^2 {\rm d}t\right)^{\frac{\tilde p}{2}}\right]\right)^{\frac{1}{\tilde p}}\left( \E\left[\left(\frac{{\rm d}\Q^{q_\cdot}}{{\rm d}\p}\right)^{\frac{\tilde p}{\tilde p-1}}\right]\right)^{\frac{\tilde p-1}{\tilde p}}<+\infty.
\end{array}
$$
Thus, \eqref{eq:3.8} follows immediately by taking the conditional mathematical expectation with respect to $\F_t$ under probability $\Q^{q_\cdot}$ on the both sides of \eqref{eq:3.11}.

To complete the proof of \cref{thm:3.2}, it remains to verify that there exists a $q^*_\cdot\in \hcal$ such that
\begin{equation}\label{eq:3.12}
Y_\cdot=\hat Y^{q^*_\cdot}_\cdot={\rm e}^{-\int_0^\cdot q^*_s{\rm d}s}Y^{q^*_\cdot}_\cdot.
\end{equation}
The following proof will be divided into three steps.

{\bf Step 1.} In this step we will prove that $Y_\cdot$ admits a desired lower bound, see \eqref{eq:3.17} below. For each $t\in\T$, we define
$$
a_t:=\frac{\tilde f(t,Y_t,Z_t)-\tilde f(t,0,Z_t)}{Y_t}{\bf 1}_{\{Y_t>0\}}\in\R
$$
and
$$
b_t:=\frac{\big(\tilde f(t,0,Z_t)-\tilde f(t,0,0)\big)Z_t^\top}{|Z_t|^2}{\bf 1}_{\{|Z_t|\neq 0\}}\in \R^d.
$$
It follows from assumption \ref{A3} that $|a_\cdot|\leq \tilde\beta$, $|b_\cdot|\leq \gamma$ and
\begin{equation}\label{eq:3.13}
\tilde f(t,Y_t,Z_t)=a_tY_t+Z_t b_t.
\end{equation}
Furthermore, define a probability measure $\bar\Q$ which is equivalent to $\p$ by
$$
\frac{{\rm d}\bar\Q}{{\rm d}\p}:=\exp\left(
\int_0^T b_s^\top B_s-\frac{1}{2}\int_0^T |b_s|^2 {\rm d}s\right).
$$
Then, both $\frac{{\rm d}\bar\Q}{{\rm d}\p}$ and $\frac{{\rm d}\p}{{\rm d}\bar\Q}$ admit moments of any order. According to Girsanov's theorem, the process
$$
\bar B_t:=B_t-\int_0^t b_s{\rm d}s,\ \ t\in \T
$$
is an $(\F_t)$-adapted standard Brownian motion, and in light of \eqref{eq:3.13}, BSDE \eqref{eq:3.2} can be written as
$$
Y_t=\xi+\int_t^T \left(f(s,Y_s)+\bar f(s,Y_s)+a_sY_s\right){\rm d}s-\int_t^T Z_s {\rm d}\bar B_s, \ \ t\in\T.
$$
Then, letting $\bar\xi:={\rm e}^{\int_0^T a_s{\rm d}s}\xi$, $\bar Y_\cdot:={\rm e}^{\int_0^\cdot a_s{\rm d}s}Y_\cdot$ and $\bar Z_\cdot:={\rm e}^{\int_0^\cdot a_s{\rm d}s}Z_\cdot$, by It\^{o}'s formula we obtain
\begin{equation}\label{eq:3.14}
\bar Y_t=\bar\xi+\int_t^T \bar g(s,\bar Y_s) {\rm d}s-\int_t^T \bar Z_s {\rm d}\bar B_s, \ \ t\in\T
\end{equation}
with
$$
\bar g(t,y):={\rm e}^{\int_0^t a_s{\rm d}s}\left(f(t,{\rm e}^{-\int_0^t a_s{\rm d}s}y)+\bar f(t,{\rm e}^{-\int_0^t a_s{\rm d}s}y)\right),\ \ (\omega,t,y)\in\Omega\times\T\times\R_{+}.
$$
From assumption \ref{A1} along with nonnegativity of $\bar f$ and boundedness of $a_\cdot$, it follows that $\as$, for each $t\in\T$ and $y\in\R_+$,
\begin{equation}\label{eq:3.15}
\bar g(t,y)\geq {\rm e}^{\int_0^t a_s{\rm d}s}f(t,{\rm e}^{-\int_0^t a_s{\rm d}s}y)\geq  {\rm e}^{\int_0^t a_s{\rm d}s}(\alpha_t+\varphi(y))\geq \bar\alpha_t+\bar\varphi(y)\geq \bar c +\bar\varphi(y)
\end{equation}
with
$$
\bar\alpha_\cdot:={\rm e}^{-\tilde\beta t}\alpha_t\geq \bar c:={\rm e}^{-\tilde\beta T}c\ \ \ {\rm and}\ \ \ \bar\varphi(y):={\rm e}^{-\tilde\beta T}\varphi(y)\in\Sbb.
$$
Now, let $\bar\varphi_{\bar c}(\cdot):=\bar c+\bar\varphi(\cdot)$. For each $n\geq 1$ and $t\in\T$, define the following stopping time
$$
\tau_n^t:=\inf\left\{u\geq t\left.| \int_0^u \left|\frac{\bar Z_s}{\bar\varphi_{\bar c}(\bar Y_s)}\right|^2\right.{\rm d}s\geq n\right\}\wedge T
$$
and the following function
$$
\Phi_{\bar c}(u):=\int_0^u \frac{1}{\bar\varphi_{\bar c}(y)}{\rm d}y=\int_0^u \frac{1}{\bar c+\bar\varphi(y)}{\rm d}y,\ \ u\in\R_+.
$$
In light of \eqref{eq:3.14} and \eqref{eq:3.15}, using It\^{o}'s formula for $\Phi_{\bar c}(\bar Y_t)$ yields that for each $n\geq 1$ and $t\in\T$,
\begin{equation}\label{eq:3.16}
\begin{array}{lll}
\Dis \Phi_{\bar c}(\bar Y_t)&=& \Dis \Phi_{\bar c}(\bar Y_{\tau_n^t})+\int_t^{\tau_n^t} \left(
\frac{\bar g(s,\bar Y_s)}{\bar\varphi_{\bar c}(\bar Y_s)}+\frac{\bar\varphi'(\bar Y_s)}{2(\bar\varphi_{\bar c}(\bar Y_s))^2} \right){\rm d}s-\int_t^{\tau_n^t} \frac{\bar Z_s}{\bar\varphi_{\bar c}(\bar Y_s)} {\rm d}\bar B_s\\
&\geq & \Dis \E_{\bar\Q}\left[\left.\Phi_{\bar c}(\bar Y_{\tau_n^t})+(\tau_n^t-t)\right|\F_t\right].
\end{array}
\end{equation}
Since $\Phi_{\bar c}(0)=0$ and for each $u>0$, $\Phi'_{\bar c}(u)>0$ and $\Phi''_{\bar c}(u)\leq 0$, we have $\Phi_{\bar c}(\cdot)\in \Sbf$ and, in light of \eqref{eq:2.2},
$$
\Phi_{\bar c}(u)\leq \Phi_{\bar c}(1)+\Phi_{\bar c}(1)u,\ \ u\in\R_+.
$$
On the other hand, by H\"{o}lder's inequality we deduce that
$$
\E\left[\sup\limits_{t\in\T}|\bar Y_t|^{\tilde p}\right]\leq {\rm e}^{\tilde p\tilde\beta T}\E\left[\sup\limits_{t\in\T}|Y_t|^{\tilde p}\right]\leq {\rm e}^{\tilde p\tilde\beta T} \left(\E\left[\sup\limits_{t\in\T}|Y_t|^p \right]\right)^{\frac{\tilde p}{p}}<+\infty
$$
and then
$$
\E_{\bar\Q}\left[\sup\limits_{t\in\T}|\bar Y_t|\right]\leq
\left(\E\left[\sup\limits_{t\in\T}|\bar Y_t|^{\tilde p}\right]\right)^{\frac{1}{\tilde p}}\left( \E\left[\left(\frac{{\rm d}\bar\Q}{{\rm d}\p}\right)^{\frac{\tilde p}{\tilde p-1}} \right]\right)^{\frac{\tilde p-1}{\tilde p}}<+\infty.
$$
In light of the last three inequalities, we can send $n\To\infty$ in \eqref{eq:3.16} to obtain that for each $t\in\T$,
$$
\Phi_{\bar c}(\bar Y_t)\geq \E_{\bar\Q}\left[\Phi_{\bar c}(\bar\xi)|\F_t\right]+T-t
$$
and then
\begin{equation}\label{eq:3.17}
Y_t\geq {\rm e}^{-\int_0^t a_s{\rm d}s} \Phi_{\bar c}^{-1}\left(\E_{\bar\Q}\left[\Phi_{\bar c}(\bar\xi)|\F_t\right]+T-t\right)
\geq {\rm e}^{-\tilde\beta T}\Phi_{\bar c}^{-1}\left(\eta+T-t\right)
\end{equation}
with
\begin{equation}\label{eq:3.18}
\eta:=\min\limits_{t\in\T}\E_{\bar\Q}\left[\Phi_{\bar c}(\bar\xi)|\F_t\right]
\in\R_{++}.
\end{equation}

{\bf Step 2.} In this step, we pick $q^*_\cdot:=f'(\cdot,Y_\cdot)\geq 0$ and prove that $q^*_\cdot\in\hcal$. First, it follows from assumption \ref{A1} and the definition of $\bar\varphi_{\bar c}(\cdot)$ that for each $t\in\T$,
$$
q^*_t=f'(t,Y_t)\leq \lambda \varphi'(Y_t)=\lambda {\rm e}^{\tilde\beta T} \bar\varphi'_{\bar c}(Y_t).
$$
Then, in light of \eqref{eq:3.17} and the definition of $\Phi_{\bar c}(\cdot)$ along with the nonincreasing property of $\bar\varphi'_{\bar c}(\cdot)$ and the nondecreasing property of $\bar\varphi_{\bar c}(\cdot)$, we can deduce that
$$
\begin{array}{lll}
\Dis \int_0^T q^*_t {\rm d}t&\leq &\Dis  \lambda {\rm e}^{\tilde\beta T} \int_0^T \bar\varphi'_{\bar c}(Y_t) {\rm d}t\leq \lambda {\rm e}^{\tilde\beta T} \int_0^T \bar\varphi'_{\bar c}\left({\rm e}^{-\tilde\beta T}\Phi_{\bar c}^{-1}\left(\eta+T-t\right)\right) {\rm d}t\vspace{0.1cm}\\
&=& \Dis \lambda {\rm e}^{\tilde\beta T}\int_{\Phi_{\bar c}^{-1}(\eta)}^{\Phi_{\bar c}^{-1}(\eta+T)}\bar\varphi'_{\bar c}\left({\rm e}^{-\tilde\beta T}u\right)\frac{1}{\bar\varphi_{\bar c}(u)}{\rm d}u\vspace{0.1cm}\\
&\leq & \Dis \lambda {\rm e}^{2\tilde\beta T}\int_{\Phi_{\bar c}^{-1}(\eta)}^{\Phi_{\bar c}^{-1}(\eta+T)}\bar\varphi'_{\bar c}\left({\rm e}^{-\tilde\beta T}u\right)\frac{1}{\bar\varphi_{\bar c}\left({\rm e}^{-\tilde\beta T}u\right)}{\rm d}\left({\rm e}^{-\tilde\beta T}u\right)\\
&=& \Dis \lambda {\rm e}^{2\tilde\beta T}\left[\ln \bar\varphi_{\bar c}\left({\rm e}^{-\tilde\beta T}\Phi_{\bar c}^{-1}(\eta+T)\right)-\ln \bar\varphi_{\bar c}\left({\rm e}^{-\tilde\beta T}\Phi_{\bar c}^{-1}(\eta)\right)\right].
\end{array}
$$
By picking $\check p\in (\bar p,\hat p)$ and using H\"{o}lder's inequality, we have
\begin{equation}\label{eq:3.19}
\begin{array}{lll}
\Dis \E\left[{\rm e}^{\bar p\int_0^T q^*_t {\rm d}t}\right]&=&\Dis \Dis\E_{\bar\Q}\left[{\rm e}^{\bar p\int_0^T q^*_t {\rm d}t}\frac{{\rm d}\p}{{\rm d}\bar\Q}\right]\leq \E_{\bar\Q}\left[\left(\frac{\bar\varphi_{\bar c}\left({\rm e}^{-\tilde\beta T}\Phi_{\bar c}^{-1}(\eta+T)\right)}{\bar\varphi_{\bar c}\left({\rm e}^{-\tilde\beta T}\Phi_{\bar c}^{-1}(\eta)\right)}\right)^{\bar p\lambda {\rm e}^{2\tilde\beta T}}\frac{{\rm d}\p}{{\rm d}\bar\Q}\right]\vspace{0.1cm}\\
&\leq & \Dis \left\{\E_{\bar\Q}\left[\left(\frac{\bar\varphi_{\bar c}\left({\rm e}^{-\tilde\beta T}\Phi_{\bar c}^{-1}(\eta+T)\right)}{\bar\varphi_{\bar c}\left({\rm e}^{-\tilde\beta T}\Phi_{\bar c}^{-1}(\eta)\right)}\right)^{\check p\lambda {\rm e}^{2\tilde\beta T}}\right]\right\}^{\frac{\bar p}{\check p}}
\left\{\E_{\bar\Q}\left[\left(\frac{{\rm d}\p}{{\rm d}\bar\Q}\right)^{\frac{\check p}{\check p-\bar p}}\right]\right\}^{\frac{\check p-\bar p}{\check p}}.
\end{array}
\end{equation}
By \cref{lem:2.2}, we have for $\eta\geq \Phi_{\bar c}(1)$,
\begin{equation}\label{eq:3.20}
\bar\varphi_{\bar c}\left({\rm e}^{-\tilde\beta T}\Phi_{\bar c}^{-1}(\eta+T)\right)\leq 2{\rm e}^{T(\varphi(1)+\bar c)}\bar\varphi_{\bar c}\left({\rm e}^{-\tilde\beta T}\Phi_{\bar c}^{-1}(\eta)\right).
\end{equation}
Since
$$
\bar\varphi_{\bar c}\left({\rm e}^{-\tilde\beta T}\Phi_{\bar c}^{-1}(\eta)\right)\geq {\rm e}^{-\tilde\beta T}\bar\varphi_{\bar c}\left(\Phi_{\bar c}^{-1}(\eta)\right),
$$
and both functions $\bar\varphi_{\bar c}(\cdot)$ and $\Phi_{\bar c}^{-1}(\cdot)$ are nondecreasing,  in view of  \eqref{eq:3.18} and \eqref{eq:3.20},  we conclude that there is a constant $C>0$ depending only on $(\tilde\beta,\bar c,T,\varphi(1),\check p,\lambda)$ such that
\begin{equation}\label{eq:3.21}
\begin{array}{l}
\Dis \E_{\bar\Q}\left[\left(\frac{\bar\varphi_{\bar c}\left({\rm e}^{-\tilde\beta T}\Phi_{\bar c}^{-1}(\eta+T)\right)}{\bar\varphi_{\bar c}\left({\rm e}^{-\tilde\beta T}\Phi_{\bar c}^{-1}(\eta)\right)}\right)^{\check p\lambda {\rm e}^{2\tilde\beta T}}\right]\leq \Dis C \E_{\bar\Q}\left[\left(\frac{1}{\bar\varphi_{\bar c}\left(\Phi_{\bar c}^{-1}(\eta)\right)}\right)^{\check p\lambda {\rm e}^{2\tilde\beta T}}{\bf 1}_{0<\eta<\Phi_{\bar c}(1)}\right]+C\vspace{0.1cm}\\
\ \ \leq \Dis C \E_{\bar\Q}\left[\max_{t\in\T}\left(\frac{1}{\bar\varphi_{\bar c}\left(\Phi_{\bar c}^{-1}\left(\E_{\bar\Q}\left[\Phi_{\bar c}(\bar\xi)|\F_t\right]\right)\right)}\right)^{\check p\lambda {\rm e}^{2\tilde\beta T}}\right]+C.
\end{array}
\end{equation}
Since $\Phi_{\bar c}(\cdot),\bar\varphi_{\bar c}(\cdot)\in \Sbf$, it is not difficult to verify that the function  $\left(\frac{1}{\bar\varphi_{\bar c}\left(\Phi_{\bar c}^{-1}(\cdot)\right)}\right)^{\check p\lambda {\rm e}^{2\tilde\beta T}/2}$ is convex on $\R_{++}$. Then, by Jensen's inequality we have
$$
\left(\frac{1}{\bar\varphi_{\bar c}\left(\Phi_{\bar c}^{-1}\left(\E_{\bar\Q}\left[\Phi_{\bar c}(\bar\xi)|\F_t\right]\right)\right)}\right)^{\check p\lambda {\rm e}^{2\tilde\beta T}}\leq \left(\E_{\bar\Q}\left[\left.\left(\frac{1}{\bar\varphi_{\bar c}(\bar\xi)}\right)^{\check p\lambda {\rm e}^{2\tilde\beta T}/2}\right|\F_t\right]\right)^2.
$$
Furthermore, by Doob's maximal inequality on sub-martingales and H\"{o}lder's inequality we obtain
\begin{equation}\label{eq:3.22}
\begin{array}{l}
\Dis \E_{\bar\Q}\left[\max_{t\in\T}\left(\frac{1}{\bar\varphi_{\bar c}\left(\Phi_{\bar c}^{-1}\left(\E_{\bar\Q}\left[\Phi_{\bar c}(\bar\xi)|\F_t\right]\right)\right)}\right)^{\check p\lambda {\rm e}^{2\tilde\beta T}}\right]\leq 4\E_{\bar\Q}\left[\left(\frac{1}{\bar\varphi_{\bar c}(\bar\xi)}\right)^{\check p\lambda {\rm e}^{2\tilde\beta T}}\right]\vspace{0.1cm}\\
\ \ \Dis =4\E\left[\left(\frac{1}{\bar\varphi_{\bar c}(\bar\xi)}\right)^{\check p\lambda {\rm e}^{2\tilde\beta T}}\frac{{\rm d}\bar\Q}{{\rm d}\p}\right]\leq
4\left\{\E\left[\left(\frac{1}{\bar\varphi_{\bar c}(\bar\xi)}\right)^{\hat p\lambda {\rm e}^{2\tilde\beta T}}\right]\right\}^{\frac{\check p}{\hat p}}
\left\{\E\left[\left(\frac{{\rm d}\bar\Q}{{\rm d}\p}\right)^{\frac{\hat p}{\hat p-\check p}}\right]\right\}^{\frac{\hat p-\check p}{\hat p}}.
\end{array}
\end{equation}
On the other hand, since $\bar\varphi(\cdot)\in \Sbf$, we have
$$
\bar\varphi_{\bar c}(\bar\xi)={\rm e}^{-\tilde\beta T}\varphi\left({\rm e}^{\int_0^T a_s{\rm d}s}\xi\right)+\bar c
\leq {\rm e}^{-\tilde\beta T}\varphi\left({\rm e}^{\tilde\beta T}\xi\right)+\bar c
\leq \varphi\left(\xi\right)+\bar c
$$
and
$$
\bar\varphi_{\bar c}(\bar\xi)={\rm e}^{-\tilde\beta T}\varphi\left({\rm e}^{\int_0^T a_s{\rm d}s}\xi\right)+\bar c
\geq {\rm e}^{-\tilde\beta T}\varphi\left({\rm e}^{-\tilde\beta T}\xi\right)+\bar c \geq {\rm e}^{-2\tilde\beta T}\varphi\left(\xi\right)+\bar c.
$$
Thus, if $\frac{1}{\varphi(\xi)+c}\in L^{\hat p\lambda {\rm e}^{2\tilde\beta T}}(\R_+)$, then it follows from \eqref{eq:3.19}, \eqref{eq:3.21} and \eqref{eq:3.22} that
$${\rm e}^{\int_0^T q_s{\rm d}s}\in L^{\bar p}(\R_+).$$
It remains to show that $(f^*(t,q^*_t))_{t\in\T}\in \lcal^p(\R_+)$. In fact, in light of \eqref{eq:3.3} and the definition of $q^*_\cdot$, similarly as in \cref{lem:2.1} we obtain that for each $t\in\T$,
\begin{equation}\label{eq:3.23}
f(t,Y_t)=f^*(t,q^*_t)+q^*_tY_t.
\end{equation}
Then, in light of \eqref{eq:3.23} and assumption \ref{A1}, we have for each $t\in\T$,
$$
0\leq f^*(t,q^*_t)\leq f(t,Y_t)\leq \bar\alpha_t+\beta Y_t.
$$
As a result, $f^*(\cdot,q^*_\cdot)\in \lcal^p(\R_+)$ due to $\bar\alpha_\cdot\in \lcal^p(\R_+)$ and $Y_\cdot\in \scal^p(\R_+)$.

{\bf Step 3.} In this step we prove that for each $t\in\T$,
\begin{equation}\label{eq:3.24}
Y_t={\rm e}^{-\int_0^t q^*_s {\rm d}s}Y_t^{q^*_\cdot}.
\end{equation}
In fact, it follows from \eqref{eq:3.23} that $(Y_\cdot,Z_\cdot)$ solves the following BSDE:
$$
Y_t=\xi+\int_t^T \left(q^*_sY_s+f^*(s,q^*_s)+\bar f(s,Y_s)+\tilde f(s,Y_s,Z_s)\right){\rm d}s-\int_t^T Z_s {\rm d}B_s, \ \ t\in\T.
$$
Define for each $t\in\T$,
$$
\check Y_t:={\rm e}^{\int_0^t q^*_s {\rm d}s}Y_t\ \ {\rm and}\ \
\check Z_t:={\rm e}^{\int_0^t q^*_s {\rm d}s}Z_t.
$$
Then, by It\^{o}'s formula we know that $(\check Y_\cdot,\check Z_\cdot)$ solves BSDE \eqref{eq:3.6} with $q_\cdot:=q^*_\cdot$. On the other hand,
by H\"{o}lder's inequality we deduce that, in light of $q^*_\cdot\in\hcal$,
$$
\E\left[\sup\limits_{t\in\T}|\check Y_t|^{\tilde p}\right]\leq \E\left[{\rm e}^{\tilde p \int_0^T q^*_s{\rm d}s}\sup\limits_{t\in\T}|Y_t|^{\tilde p}\right]\leq \left(\E\left[ {\rm e}^{\bar p \int_0^T q^*_s{\rm d}s}\right]\right)^{\frac{\tilde p}{\bar p}}\left(\E\left[\sup\limits_{t\in\T}|Y_t|^p \right]\right)^{\frac{\tilde p}{p}}<+\infty
$$
and
$$
\begin{array}{lll}
\Dis \E\left[\left(\int_0^T |\check Z_t|^2 {\rm d}t\right)^{\frac{\tilde p}{2}}\right]&\leq & \Dis \E\left[{\rm e}^{\tilde p \int_0^T q^*_s{\rm d}s}\left(\int_0^T |Z_t|^2 {\rm d}t\right)^{\frac{\tilde p}{2}}\right]\vspace{0.1cm}\\
&\leq & \Dis \left(\E\left[ {\rm e}^{\bar p \int_0^T q^*_s{\rm d}s}\right]\right)^{\frac{\tilde p}{\bar p}}\left(\E\left[\left(\int_0^T |Z_t|^2 {\rm d}t\right)^{\frac{p}{2}}\right]\right)^{\frac{\tilde p}{p}}\vspace{0.1cm}
< +\infty,
\end{array}
$$
which means that $(\check Y_\cdot, \check Z_\cdot)\in \scal^{\tilde p}(\R_{++})\times\mcal^{\tilde p}(\R^{1\times d})$. Thus, the uniqueness of the adapted solution of BSDE \eqref{eq:3.6} in the space of $\scal^{\tilde p}(\R_{++})\times\mcal^{\tilde p}(\R^{1\times d})$ yields that $\check Y_\cdot=Y^{q^*_\cdot}_\cdot$, which is \eqref{eq:3.24}. That is to say, \eqref{eq:3.12} holds true. The proof of \cref{thm:3.2} is then complete.
\end{proof}

\section{Sharp result on some special BSDE}
\label{sec:4-FurtherDiscussion}
\setcounter{equation}{0}

In this section, we are concerned with the following BSDE of special form:
\begin{equation}\label{eq:4.1}
Y_t=\xi+\int_t^T \left[(k_1(s))^{1-\alpha}Y_s^\alpha+k_2(s)Y_s+k_3(s)|Z_s|+k_4(s)\right]{\rm d}s-\int_t^T Z_s {\rm d}B_s, \ \ t\in\T,
\end{equation}
where $\xi\in L^p(\R_{++})$ for some $p>1$, both $\alpha\in (0,1)$ and $c\in\R_{++}$ are constants, and $k_i(\cdot)\ {i=1,2,3,4}$ are four $(\F_t)$-progressively measurable real-valued processes defined on $\Omega\times\T$ such that
$$
k_1(\cdot)\in \lcal^p(\R_+),\ \ k_2(\cdot)\in [-c,c],\ \ k_3(\cdot)\in [0,c]\ \ {\rm and}\ \ k_4(\cdot)\in\lcal^p(\R_+).
$$

The main result of this section is stated as follows.

\begin{thm}\label{thm:4.1}
BSDE \eqref{eq:4.1} admits a unique adapted solution $(Y_t,Z_t)_{t\in\T}$ in $\scal^p(\R_{++})\times\mcal^p(\R^{1\times d})$.
\end{thm}

\begin{proof} For each $x,k\in\R_+$, by Young's inequality we have
\begin{equation}\label{eq:4.2*}
k^{1-\alpha}x^\alpha\leq (1-\alpha)k+\alpha x.
\end{equation}
The existence of an adapted solution $(Y_\cdot,Z_\cdot)_{t\in\T}\in\scal^p(\R_{++})\times\mcal^p(\R^{1\times d})$ of BSDE \eqref{eq:4.1} is available in \cref{rmk:3.3} (i). Now, we prove the uniqueness by employing the $\theta$-difference method used in for example \citet{BriandHu2008PTRF} and  \citet{FanHuTang2023arXiv}.

Define $\bar \xi:={\rm e}^{\int_0^T k_2(s){\rm d}s}\xi\in L^p(\R_{++})$ and
$$
\bar Y_\cdot:={\rm e}^{\int_0^\cdot k_2(s){\rm d}s} Y_\cdot\in \scal^p(\R_{++}),\ \ \ \bar Z_\cdot:={\rm e}^{\int_0^\cdot k_2(s){\rm d}s}Z_\cdot\in \mcal^p(\R^{1\times d}),
$$
$$
\bar k_1(\cdot):={\rm e}^{\int_0^\cdot k_2(s){\rm d}s}k_1(\cdot)\in \lcal^p(\R_+),\ \ \ {\rm and}\ \ \ \bar k_4(\cdot):={\rm e}^{\int_0^\cdot k_2(s){\rm d}s}k_4(\cdot)\in \lcal^p(\R_+).
$$
It follows from It\^{o}'s formula that $(\bar Y_\cdot, \bar Z_\cdot)$ solves the following BSDE:
\begin{equation}\label{eq:4.2}
\bar Y_t=\bar\xi+\int_t^T \left[(\bar k_1(s))^{1-\alpha} \bar Y_s^\alpha+ k_3(s)|\bar Z_s|+\bar k_4(s)\right]{\rm d}s-\int_t^T \bar Z_s {\rm d}B_s, \ \ t\in\T.
\end{equation}
Let
$$
H(u):=\int_0^u \frac{1}{x^\alpha}{\rm d}x=\frac{1}{1-\alpha}u^{1-\alpha},\ \ u\in\R_{+}.
$$
Then
$$
H^{-1}(x):=\left[(1-\alpha)x\right]^{\frac{1}{1-\alpha}},\ \ x\in\R_+
$$
is the inverse of function $H(\cdot)$. Furthermore, we define $\tilde\xi:=H(\bar\xi)\in L^{\frac{p}{1-\alpha}}(\R_{++})$,
$$
\tilde Y_\cdot:=H(\bar Y_\cdot)=\frac{1}{1-\alpha}\bar Y_\cdot^{1-\alpha}\in \scal^{\frac{p}{1-\alpha}}(\R_{++}),\ \ \ \tilde Z_\cdot:=H'(\bar Y_\cdot)\bar Z_\cdot=\frac{\bar Z_\cdot}{\bar Y_\cdot^\alpha}
$$
and
$$
\tilde k_4(\cdot):=
\frac{1}{\left(1-\alpha\right)^\frac{\alpha}{1-\alpha}}\bar k_4(\cdot)\in \lcal^p(\R_+).
$$
By applying It\^{o}'s formula to $\tilde Y_\cdot$, it is not difficult to verify that
\begin{equation}\label{eq:4.3}
\tilde Y_t=\tilde\xi+\int_t^T g(s,\tilde Y_s,\tilde Z_s){\rm d}s-\int_t^T \tilde Z_s {\rm d}B_s, \ \ t\in\T,
\end{equation}
where for each $(\omega,t,y,z)\in \Omega\times\T\times\R_+\times\R^{1\times d}$,
\begin{equation}\label{eq:4.4}
g(\omega,t,y,z):=(\bar k_1(\omega,t))^{1-\alpha}+ k_3(\omega,t)|z|+\frac{\alpha}{2(1-\alpha)}\frac{|z|^2}{y}+\tilde k_4(\omega,t)\frac{1}{y^{\frac{\alpha}{1-\alpha}}}.
\end{equation}
Obviously $\as$, the function $g(\omega,t,\cdot,\cdot)$ is jointly convex on $\R_{++}\times\R^{1\times d}$.

By virtue of the $\theta$-difference method we will show that BSDE \eqref{eq:4.3} admits at most an adapted solution $(\tilde Y_\cdot,\tilde Z_\cdot)$ such that $\tilde Y_\cdot\in \scal^{\frac{p}{1-\alpha}}(\R_{++})$. In fact, suppose that both $(\tilde Y_\cdot^1,\tilde Z_\cdot^1)$ and $(\tilde Y_\cdot^2,\tilde Z_\cdot^2)$ are adapted solutions of BSDE \eqref{eq:4.3} in the space $\scal^{\frac{p}{1-\alpha}}(\R_{++})$. For each $\theta\in (0,1)$, define
$$
\hat Y^\theta_\cdot:=\frac{\tilde Y^1_\cdot-\theta \tilde Y^2_\cdot}{1-\theta}\ \ \ {\rm and}\ \ \ \hat Z^\theta_\cdot:=\frac{\tilde Z^1_\cdot-\theta \tilde Z^2_\cdot}{1-\theta}.
$$
Observe that $\as$, for each $(y_1,y_2,z_1,z_2)\in\R_+\times\R_+\times\R^{1\times d}\times\R^{1\times d}$,
$$
\begin{array}{l}
\Dis {\bf 1}_{\{y_1>\theta y_2\}}\left(g(t,y_1,z_1)-\theta g(t,y_2,z_2)\right)\vspace{0.1cm}\\
\ \ \Dis ={\bf 1}_{\{y_1>\theta y_2\}}\left[g\left(t,~\theta y_2+(1-\theta)\frac{y_1-\theta y_2}{1-\theta},~\theta z_2+(1-\theta)\frac{z_1-\theta z_2}{1-\theta}\right)-\theta g(t,y_2,z_2)\right]\vspace{0.1cm}\\
\ \ \Dis \leq (1-\theta){\bf 1}_{\{y_1>\theta y_2\}}g\left(t,\ \left(\frac{y_1-\theta y_2}{1-\theta}\right)^+, \frac{z_1-\theta z_2}{1-\theta}\right).
\end{array}
$$
Using It\^{o}-Tanaka's formula, we conclude that there exists an $(\F_t)$-progressively measurable $\R_+$-valued process $(\kappa^\theta_t)_{t\in\T}$ depending on $\theta$ such that
\begin{equation}\label{eq:4.5}
(\hat Y_t^\theta)^+=\tilde \xi+\int_t^T \left[{\bf 1}_{\{\hat Y_s^\theta>0\}}g(s,(\hat Y_s^\theta)^+,\hat Z_s^\theta)-\kappa^\theta_s\right]{\rm d}s-\int_t^T {\rm d}L_s^\theta -\int_t^T {\bf 1}_{\{\hat Y_s^\theta>0\}}\hat Z_s^\theta {\rm d}B_s, \ \ t\in\T,
\end{equation}
where $L^\theta_\cdot$ denotes the local time of $\hat Y^\theta_\cdot$ at $0$, which is a nonnegative nondecreasing process. Furthermore, let
$$
\bar Y^\theta_\cdot:=H^{-1}((\hat Y^\theta_\cdot)^+)=\left[(1-\alpha) (\hat Y^\theta_\cdot)^+\right]^{\frac{1}{1-\alpha}}\in \scal^p(\R_+)
$$
and
$$
\bar Z^\theta_\cdot:=(H^{-1})'((\hat Y^\theta_\cdot)^+){\bf 1}_{\{\hat Y_\cdot^\theta>0\}}\hat Z^\theta_\cdot=\frac{1}{1-\alpha}\left[(1-\alpha)(\hat Y^\theta_\cdot)^+\right]^{\frac{\alpha}{1-\alpha}}\hat Z^\theta_\cdot.
$$
In light of the form of \eqref{eq:4.5} along with those arguments from \eqref{eq:4.2} to \eqref{eq:4.4}, applying It\^{o}'s formula to $\bar Y^\theta_\cdot$ yields that for each $t\in\T$,
\begin{equation}\label{eq:4.6}
\bar Y_t^{\theta}=\bar\xi+\int_t^T \left[(\bar k_1(s))^{1-\alpha} (\bar Y_s^\theta)^\alpha+ k_3(s)|\bar Z_s^\theta|+{\bf 1}_{\{\hat Y_s^\theta>0\}}\bar k_4(s)-\bar\kappa^\theta_s\right]{\rm d}s-\int_t^T {\rm d}\bar L^\theta_s -\int_t^T \bar Z_s^\theta {\rm d}B_s,
\end{equation}
where
$$
\bar\kappa^\theta_\cdot:=\left[(1-\alpha)(\hat Y^\theta_\cdot)^+\right]^{\frac{\alpha}{1-\alpha}}\kappa^\theta_\cdot\ \ \ {\rm and}\ \ \ \bar L^\theta_\cdot:=\int_0^\cdot \left[(1-\alpha)(\hat Y^\theta_s)^+\right]^{\frac{\alpha}{1-\alpha}}{\rm d} L^\theta_s.
$$
In the sequel, note that $\bar\xi\in L^p(\R_{++})$, $\bar Y^\theta_\cdot\in\scal^p(\R_+)$, $\bar k_1(\cdot)\in \lcal^p(\R_+)$, $k_3(\cdot)\in [0, c]$, $\bar k_4(\cdot)\in \lcal^p(\R_+)$, $\bar\kappa^\theta_\cdot\in \R_+$, $\bar L^\theta_\cdot$ is a nonnegative nondecreasing process and in light of \eqref{eq:4.2*},
$$
(\bar k_1(s))^{1-\alpha} (\bar Y_s^\theta)^\alpha\leq (1-\alpha)\bar k_1(s)+\alpha \bar Y_s^\theta,\ \ \ s\in\T.
$$
Using a standard argument to establish a priori estimates for the adapted solution of BSDE \eqref{eq:4.6}, see Appendix for details, we deduce that there is a constant $C>0$ independent of $\theta$ such that for each $t\in\T$ and $\theta\in (0,1)$,
$$
|\bar Y^\theta_t|^p\leq C\E\left[\left.|\bar\xi|^p+\left(\int_0^T \left( \bar k_1(s)+\bar k_4(s)\right) {\rm d}s\right)^p\right|\F_t\right]+C,
$$
and then
$$
(\tilde Y^1_t-\theta \tilde Y^2_t)^+\leq \frac{1-\theta}{1-\alpha} \left\{
C\E\left[\left.|\bar\xi|^p+\left(\int_0^T \left( \bar k_1(s)+\bar k_4(s)\right) {\rm d}s\right)^p\right|\F_t\right]+C\right\}^{\frac{1-\alpha}{p}}.
$$
Sending $\theta\to 1$ in the last inequality yields $\tilde Y^1_\cdot\leq \tilde Y^2_\cdot$.  In the same way,  we have $\tilde Y^2_\cdot\leq \tilde Y^1_\cdot$, and therefore, $\tilde Y^2_\cdot=\tilde Y^1_\cdot$. That is to say, BSDE \eqref{eq:4.3} admits at most an adapted solution $(\tilde Y_t,\tilde Z_t)_{t\in\T}$ such that $\tilde Y_\cdot\in \scal^{\frac{p}{1-\alpha}}(\R_{++})$.

Finally, from those arguments from \eqref{eq:4.2*} and \eqref{eq:4.4},  we see that BSDE \eqref{eq:4.1} admits at most an adapted solution $(Y_\cdot,Z_\cdot)\in\scal^p(\R_{++})\times\mcal^p(\R^{1\times d})$.
The proof is complete.
\end{proof}

\begin{rmk}\label{rmk:4.1*}
We have the following three remarks.
\begin{itemize}
\item [(i)] The conclusion of \cref{thm:4.1} remains true when the expression $|Z_s|$ appearing in BSDE \eqref{eq:4.1} is replaced with $f(Z_s)$ for a positive homogeneous and convex function $f(\cdot):\R^d\To \R_+$. In fact, in the proof of \cref{thm:4.1} one only needs to replace the expressions $|\bar Z_s|$, $k_3(\omega,t)|z|$ and $|\bar Z_s^\theta|$ appearing in \eqref{eq:4.2}, \eqref{eq:4.4} and \eqref{eq:4.6} with $f(\bar Z_s)$, $k_3(\omega,t)f(z)$ and $f(\bar Z_s^\theta)$, respectively.

\item [(ii)] By virtue of the relevant result on the $L^1$ solution obtained for example in \citet{BriandDelyonHu2003SPA} and \citet{FanHuTang2023arXiv} one can verify that when $k_3(\cdot)\equiv 0$ and $p=1$, BSDE \eqref{eq:4.1} admits a unique adapted solution $(Y_t,Z_t)_{t\in\T}$ such that $Y_\cdot$ belongs to class (D).

\item [(iii)] For the case of $\essinf_{t\in\T}k_4(t)=0$, the conclusion of \cref{thm:4.1} is better than that of \cref{thm:3.2} since the integrability condition on $1/\varphi(\xi)$ in \cref{thm:3.2} is moved away in \cref{thm:4.1}. However, \cref{thm:4.1} is only able to tackle BSDEs with generators of certain form.
\end{itemize}
\end{rmk}

Finally,  we illustrate \cref{thm:4.1}, by the following typical financial example.

\begin{ex}\label{ex:4.2}
The Kreps-Porteus (or Epstein-Zin) utility of an endowment studied in \citet{ElKarouiPengQuenez1997MF} (see also \citet{Skiadas2013JPE}, \citet{Fan2016SPA}, \citet{Xing2017FS}, \citet{Wang2019IME} and \citet{KublerSeldenWei2020JET}) can be defined by the adapted solution of the following BSDE:
\begin{equation}\label{eq:4.7}
Y_t=\xi+\int_0^T g(Y_s){\rm d}s-\int_t^T Z_s {\rm d}s,\ \ t\in\T,
\end{equation}
where $\xi\in L^p(\R_+)$ for some $p>1$, and for each $y\in\R_+$,
\begin{equation}\label{eq:4.8}
g(y):=\frac{\rho}{\beta}\frac{c^\rho-y^\rho}{y^{\rho-1}}
=\frac{\rho}{\beta}(c^\rho y^{1-\rho}-y)
\end{equation}
with $\beta,c\geq 0$ and $0\neq \rho\leq 1$. It is easily verified that whether the case of $\rho=1$ or the case of $\rho<0$, the generator $g$ satisfies the monotonicity condition. Consequently, by classical results it is concluded that BSDE \eqref{eq:4.7} admits a unique adapted solution $(Y_\cdot, Z_\cdot)\in \scal^p(\R_+)\times\mcal^p(\R^{1\times d})$.

For the case of $\rho\in (0,1)$, the generator $g$ is of Peano-type and all previous existing results fail to answer whether BSDE \eqref{eq:4.7} admits a unique adapted solution $(Y_\cdot, Z_\cdot)$ in the space $\scal^p(\R_{++})\times\mcal^p(\R^{1\times d})$ for  $\xi\in L^p(\R_{++})$. However, \cref{thm:4.1} gives an affirmative answer to this question.
More generally, \cref{thm:4.1}  yields the affirmative answer when the constant $c$ in \eqref{eq:4.8} is replaced with a general consumption process $(c_t)_{t\in\T}\in \lcal^p(\R_+)$.

Moreover, in light of \cref{thm:4.1} along with  \cref{rmk:4.1*} (ii), we see that for the case of $\rho\in (0,1)$, when $\xi\in L^1(\R_{++})$ and the constant $c$ in \eqref{eq:4.8} is replaced with a general consumption process $(c_t)_{t\in\T}\in \lcal^1(\R_+)$, BSDE \eqref{eq:4.7} admits a unique adapted solution $(Y_\cdot, Z_\cdot)$ such that $Y_\cdot$ belongs to class (D).
\end{ex}

\appendix
\section{An a priori estimate}
\renewcommand{\appendixname}{}

In this section, we shall give the a priori estimate employed in the proof of \cref{thm:4.1}.

\begin{pro}\label{pro:5.1}
Suppose that $\xi\in L^p(\R_+)$ for $p>1$, $(L_t)_{t\in\T}$ is an $(\F_t)$-progressively measurable nondecreasing process with $L_0=0$, and the function $g(\omega,t,y,z):\Omega\times\T\times\R_+\times \R^{1\times d}$ is $(\F_t)$-progressively measurable for each $(y,z)\in \R_+\times \R^{1\times d}$ such that $\as$,
$$
g(\omega,t,y,z)\leq f_t(\omega)+\mu |y|+\lambda |z|, \quad (y,z)\in \R_+\times \R^{1\times d},
$$
with $f_\cdot\in \lcal^p(\R_+)$, and $\mu,\lambda\geq 0$ are two nonnegative constants. Let $(Y_t,Z_t)_{t\in\T}$ be an adapted solution of the following BSDE
$$
Y_t=\xi+\int_t^T g(s,Y_s,Z_s){\rm d}s-\int_t^T{\rm d}L_s -\int_t^T Z_s {\rm d}B_s, \ \ t\in\T.
$$
If $Y_\cdot\in\scal^p(\R_{+})$, then $Z_\cdot\in\mcal^p(\R^{1\times d})$ and there exists a constant $C_p>0$ depending only on $p$ such that for each $a\geq \mu+\lambda^2$ and $t\in\T$, we have
$$
\E\left[\left.\left(\int_0^T {\rm e}^{2as}|Z_s|^2 {\rm d}s\right)^{\frac{p}{2}}\right|\F_t\right]
\leq C_p \E\left[\left.\sup\limits_{t\in\T} \left({\rm e}^{apt}|Y_t|^p\right)+\left(\int_0^T {\rm e}^{as}f_s {\rm d}s\right)^p\right|\F_t\right].
$$
Moreover, there exists a constant $\bar C_p>0$ depending only on $p$ such that for each $a\geq \mu+\lambda^2/[1\wedge (p-1)]$ and $t\in\T$, we have
$$
\E\left[\left.\sup\limits_{t\in\T} \left({\rm e}^{apt}|Y_t|^p\right)+\left(\int_0^T {\rm e}^{2as}|Z_s|^2 {\rm d}s\right)^{\frac{p}{2}}\right|\F_t\right]
\leq \bar C_p \E\left[\left. {\rm e}^{apT}|\xi|^p +\left(\int_0^T {\rm e}^{as}f_s {\rm d}s\right)^p\right|\F_t\right].
$$
\end{pro}

\begin{proof}
Observe that under the conditions of \cref{pro:5.1} we have for each $t\in\T$,
$$
{\rm sgn}(Y_t)g(t,Y_t,Z_t)\leq f_t+\mu |Y_t|+\lambda |Z_t|, \quad 2\int_t^T Y_s {\rm d}L_s\leq 0, \quad
\text{\ and\ } \quad
p\int_t^T |Y_s|^{p-1} {\rm sgn}(Y_s) {\rm d}L_s\leq 0.
$$
The desired conclusions follow immediately, using the conditional mathematical expectation with respect to $\F_t$ instead of the original mathematical expectation in the proof procedure of Lemma 3.1 and Proposition 3.2 in \citet{BriandDelyonHu2003SPA}. More details are omitted here. The proof is complete.
\end{proof}



\setlength{\bibsep}{2pt}

\end{document}